\documentclass[11pt]{amsart}

\usepackage{amssymb}
\usepackage{amsmath}
\usepackage{amscd}
\usepackage{float}
\usepackage{graphicx}
\usepackage{amsfonts}
\usepackage{pb-diagram}

\newtheorem{thm}{Theorem}

\newtheorem{lemma}{Lemma}

\newtheorem{defn}{Definition}
\newtheorem{remark}{Remark}

\begin{document}

\title[Braid equivalence in 3-manifolds with rational surgery]
  {Braid equivalence in 3-manifolds with rational surgery description}
	
\author{Ioannis Diamantis}
\address{Department of Mathematics,
National Technical University of Athens,
Zografou Campus,
GR-15780 Athens, Greece.}
\email{diamantis@math.ntua.gr}

\author{Sofia Lambropoulou}
\address{Department of Mathematics,
National Technical University of Athens,
Zografou Campus,
GR-15780 Athens, Greece.}
\email{sofia@math.ntua.gr}
\urladdr{http://www.math.ntua.gr/$\tilde{~}$sofia}



\thanks{This research  has been co-financed by the European Union (European Social Fund - ESF) and Greek national funds through the Operational Program ``Education and Lifelong Learning" of the National Strategic Reference Framework (NSRF) - Research Funding Program: THALES: Reinforcement of the interdisciplinary and/or inter-institutional research and innovation. }

\setcounter{section}{-1}

\date{}

\begin{abstract}
In this paper we describe braid equivalence for knots and links in a 3-manifold $M$ obtained by rational surgery along a framed link in $S^3$. We first prove a sharpened version of the Reidemeister theorem for links in $M$. We then give geometric formulations of the braid equivalence via mixed braids in $S^3$ using the  $L$-moves and the braid band moves.
 We finally give algebraic formulations in terms of  the mixed braid groups $B_{m,n}$ using cabling and the techniques of parting and combing for mixed braids. We also provide concrete formuli of the braid equivalence in the case where $M$ is a lens space, a Seifert manifold or a homology sphere obtained from the trefoil. The  algebraic classification of knots and links in a $3$-manifold via mixed braids is a useful tool for studying skein modules of $3$-manifolds.
\end{abstract}

\maketitle

\section{Introduction}

In the study of knots and links in $3$-dimensional spaces (such as handlebodies, knot complements, c.c.o. $3$-manifolds) it can prove very useful
to take an approach via braids, as the use of braids provides more structure and more control on the braid  equivalence moves.

\smallbreak
In \cite{LR1} geometric braid equivalence has been given for isotopic knots and links in knot complements and in c.c.o. $3$-manifolds with integral surgery description. This was done by fixing (pointwise) a description for the knot complement
or the $3$-manifold via a closed braid $\widehat{B}$ in $S^3$. Then, knots and links in such a $3$-manifold are represented unambiguously by {\it mixed
links} in $S^3$ (Figure~\ref{A mixed link}) and braids in the manifold are represented unambiguously by geometric {\it mixed
braids} in $S^3$.  A geometric mixed braid is a braid in $S^3$ that contains $B$ as a fixed subbraid (Figure~\ref{mixedbraidsLmoves}).
 Isotopy for links in a c.c.o. 3-manifold $M$ is then translated into appropriate moves for mixed links in $S^3$, which comprise isotopy in the complement
$S^3\backslash \widehat{B}$ together with the {\it band moves} coming from the handle sliding moves in $M$,  and are related to the surgery description of $M$ (Figure~\ref{Band Moves}).
For the braid equivalence, the authors sharpened first the classic Markov theorem giving only one type of moves, the {\it $L$-moves} (Figure~\ref{mixedbraidsLmoves}), which are geometric as well as algebraic.
Then, with the use of the $L$-moves and the {\it braid band moves} (Figure~\ref{pbbm}) they formulated geometric mixed braid equivalence for knots and links in knot complements and in c.c.o. $3$-manifolds
obtained from $S^3$ by integral surgery.

Further, in \cite{LR2} the same authors provided algebraic formulations for the geometric mixed  braid equivalences in knot complements and in c.c.o. $3$-manifolds, using the {\it mixed braid groups} $B_{m,n}$ (Eq.~\ref{B} and Figure~\ref{Loops and Crossings}), introduced and studied in \cite{La2}, and the techniques of  parting and combing mixed braids (Figure~\ref{partingandcombing}). {\it Parting} a geometric mixed braid means to separate its strands into two sets: the strands of the fixed subbraid $B$ and the moving strands of the braid representing a link in the $3$-manifold. {\it Combing} a parted mixed braid means to separate the braiding of the fixed subbraid $B$ from the braiding of the moving strands (Figure~\ref{combing}). The above techniques have been also applied in \cite{HL} for links in a handlebody.

\smallbreak
Integral surgery is a special case of rational surgery and, by a classic result of topology, every c.c.o. $3$-manifold can be obtained
by surgery (integral or rational) along a framed link in $S^3$. Moreover, in the case of integral surgery the components of the framed link may be all assumed to be simple closed curves, giving rise to a  closed pure braid.
 There are $3$-manifolds which have simpler description when obtained from $S^3$
by rational surgery. Representative examples are the lens spaces $L(p,q)$ which with rational surgery description $p/q$ are obtained
from the trivial knot, while with integral surgery description a non-trivial link is needed, see for example \cite{Ro} p.322. A simpler surgery description of a $3$-manifold $M$ is expected to induce simpler algebraic expressions for the braid equivalence in $M$. As an example, compare Section~4 in
\cite{LR2} with Section~5.1 in this paper for the case of lens spaces. In this paper the braid equivalence is in the mixed braid groups $B_{1,n}$ and there is only one expression for the braid band moves. In \cite{LR2} there are many, according to the surgery coefficient of each strand of the surgery pure braid and, on top of that, combing is also needed.

\smallbreak
The purpose of this paper is to provide mixed braid equivalence, geometric as well as algebraic, for isotopic oriented links in c.c.o. $3$-manifolds obtained by rational surgery along framed links in $S^3$. We follow the setting and the techniques of \cite{LR1,LR2} and we use the results therein. More precisely, let $M$ be a c.c.o. $3$-manifold obtained by rational surgery along a framed link $\widehat{B}$ in $S^3$. Note that the surgery braid $B$ may not be assumed to be a pure braid. Let $s$ be a surgery component  of $B$ with surgery description $p/q$ consisting of $k$ strands, $s_1, \ldots, s_k$. When a  {\it geometric braid band move} on $s$ occurs, $k$ sets of $q$ new strands appear, each one running in parallel to a strand of $s$, and also a $(p,q)$-torus braid wraps around the last strand, $s_k$, $p$ times (see Figure~\ref{homtref} for an example). These moves together with the $L$-moves lead to the {\it geometric mixed braid equivalence} in $M$ (Theorem~\ref{geommarkovrat}).  The geometric braid band moves are much more complicated than in the case of integral surgery \cite{LR1}. However, a sharpened version of the Reidemeister theorem for  links in $M$ (Theorem~\ref{reidemrational}; see also \cite{LR1, Sk, Su}), whereby only one type of band moves is used in the isotopy equivalence (Figure~\ref{bmoves}), makes the proof of Theorem~\ref{geommarkovrat} lighter.

In order to move toward algebraic statements we use the notion of a $q$-strand cable and we apply the techniques and results from \cite{LR2}. A {\it $q$-strand cable} encloses a set of  $q$ new strands  arising from the performance of a geometric braid band move. We show first that parting a $q$-strand cable is equivalent to parting each strand of the cable one by one; that is, parting and cabling commute (Figure~\ref{stndpartcable}). Treating now each one of the $k$ $q$-strand cables as one thickened strand leads to the {\it parted mixed braid equivalence} (Theorem~\ref{qparted}), assuming the corresponding result in \cite{LR2}. We continue by providing algebraic expressions for parted cables (Lemma~\ref{cableloop})  and also for the loopings between the strands of $B$ and the remaining strands of the mixed braid after a braid band move is performed (Figures~\ref{cablesa}, \ref{cablesa2}). Then, we part locally the $(p,q)$-torus braid and the crossing of a parted mixed braid band move  (Figures~\ref{cables2}, \ref{cablesc}), obtaining algebraic expressions for these parts.  From the above we obtain the algebraic expression of an {\it algebraic braid band move}. Finally, we do combing through $B$ and we show that combing a $q$-strand cable is equivalent to combing each strand of the cable one by one; that is, combing and cabling commute (Figures~\ref{combcable1} and \ref{combcableproof2}). So, assuming the corresponding result in \cite{LR2}, we obtain the {\it algebraic mixed braid equivalence} for links in $M$ in terms of the mixed braid groups $B_{m,n}$, which is our main result (Theorem~\ref{algmarkov}).

The paper is organized as follows. In Section~1 we recall the setting and the essential techniques and results from \cite{LR1,LR2}. In Section~2 we prove the sharpened version of the Reidemeister theorem for knots and links in c.c.o. $3$-manifolds with rational surgery description. In Section~3 we give the geometric braid equivalence for links in such $3$-manifolds (Theorem~\ref{geommarkovrat}) and in Section~4 we give the corresponding algebraic version (Theorem~\ref{algmarkov}). Finally, in Section~5 we give the concrete algebraic expressions for the mixed braid equivalences in lens spaces, in Seifert manifolds and in homology spheres obtained from $S^3$ by rational surgery along the trefoil knot. Also, we provide some figures  (Figures~\ref{seifcomb1}--\ref{seifcomb6}) illustrating in a generic example   all the techniques developed and used in this paper  step-by-step.

\smallbreak
The algebraic mixed braid equivalence gives good control over the band moves, so our results can be applied for the study of skein modules of c.c.o. $3$-manifolds obtained by rational surgery along a framed link in $S^3$ and for constructing invariants of links in such $3$-manifolds, using algebraic means.
 For a construction of the analogues of the HOMFLYPT polynomial for knots and links in the solid torus via this approach, the reader is referred to \cite{La3}.  In a sequel paper we shall use Theorem~\ref{algmarkov} and the results of \cite{La3} in order to compute the HOMFLYPT skein module of the lens spaces $L(p,q)$ with this method.

\section{Background results in the case of integral surgery description}

In this section we recall the topological settings and results from [LR1] and [LR2] for expressing link isotopy in a c.c.o. $3$-manifold $M$ in terms of geometric and algebraic mixed braid equivalence in $S^3$. For the rest of this section we fix a c.c.o. $3$-manifold $M$, which is obtained by integral surgery
on a framed link in $S^3$ given in the form of a closed braid $\widehat{B}$. We shall denote $M=\chi_{_{\mathbb{Z}}}(S^3, \widehat{B})$.

\subsection{Mixed links and isotopy}

Let $L$ be an oriented link in $M$. Fixing $\widehat{B}$
pointwise, $L$ can be represented unambiguously by a \textit{mixed link} in $S^3$ denoted $\widehat{B}\cup L$, that is, a link in $S^{3}$ consisting of the  \textit{fixed part} $\widehat{B}$ and the \textit{moving part} $L$ that links with $\widehat{B}$. A \textit{mixed link diagram }is a diagram $\widehat{B}\cup \widetilde{L}$ of $\widehat{B}\cup L$ on the plane of $\widehat{B}$, where this plane is equipped with the top-to-bottom direction of the braid $B$, see Figure \ref{A mixed link} for an example.

\begin{figure}
\begin{center}
\includegraphics[width=1.3in]{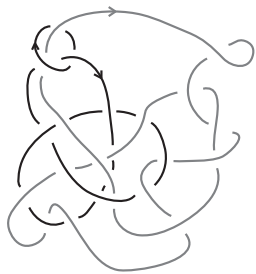}
\end{center}
\caption{ A mixed link in $S^{3}$. }
\label{A mixed link}
\end{figure}

An isotopy of $L$ in $M$ can be translated into a finite sequence of moves of the mixed link $\widehat{B} \bigcup L$ in $S^3$ as follows.
As we know, surgery along $\widehat{B}$ is realized by taking first the complement $S^3\backslash \widehat{B}$ and then attaching to it solid tori
according to the surgery description. Thus, isotopy in $M$ can be viewed as certain moves in $S^3$, namely, isotopy in $S^3\backslash \widehat{B}$ together with the band moves in $S^3$, which are similar to the second Kirby move.
Isotopy in $S^3\backslash \widehat{B}$ is realized by the classical Reidemeister moves and planar moves for the moving part together with
the {\it extended Reidemeister moves}. These are the Reidemeister II and III moves involving the fixed and the moving part of the mixed link
(cf. Definition~5.1 \cite{LR1}). A \textit{band move} is a non-isotopy move in $S^3 \backslash \widehat{B}$ that reflects isotopy in $M$ and is the band connected sum of a component, say $s$, of $L$ with the specified (from the framing) parallel curve $l$ of a surgery component, say $c$, of $\widehat{B}$. Note that $l$ bounds a disc in $M$.
 There are two types of band moves according to the orientations of the component $s$ of $L$ and of the surgery
curve $c$, as illustrated and exemplified in Figure~\ref{Band Moves}. In the $\alpha$-type the orientation of $s$ is opposite
to the orientation of $c$ (and of its parallel curve $l$), but after the performance of the move their orientations agree.
In the $\beta$-type the orientation of $s$ agrees initially with the orientation of $c$, but disagrees after the performance of the move.
Note that the two types of band moves are related by a twist of $s$ (Reidemeister I move in $S^3\backslash \widehat{B}$).

\begin{figure}
\begin{center}
\includegraphics[width=5.5in]{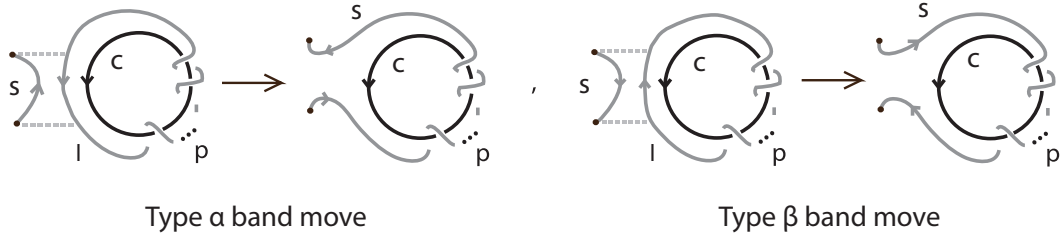}
\end{center}
\caption{ The two types of $\mathbb{Z}$-band moves. }
\label{Band Moves}
\end{figure}

The above are summarized in the following analogue of the Reidemeister theorem for oriented links in $M$.

\begin{thm}[Reidemeister for $M=\chi_{_{\mathbb{Z}}}(S^3, \widehat{B})$, Thm. 5.8 \cite{LR1}] \label{reidemintegral}
Two oriented links $L_1$, $L_2$ in  $M=\chi_{_{\mathbb{Z}}}(S^3, \widehat{B})$ are isotopic if and only if any two corresponding mixed
link diagrams of theirs, $\widehat{B} \cup \widetilde{L_1}$ and $\widehat{B} \cup \widetilde{L_2}$, differ by isotopy in $S^3 \backslash \widehat{B}$
together with a finite sequence of the two types $\alpha$ and $\beta$ of band moves.
\end{thm}

\subsection{Geometric mixed braids and the $L$-moves}

In order to translate isotopy of links in the $3$-manifold $M$ to braid equivalence, we need to introduce the notion of a geometric
mixed braid. A \textit{geometric mixed braid} related to $M=\chi_{_{\mathbb{Z}}}(S^3, \widehat{B})$ and to a link $K$ in $M$, is an element of the group $B_{m+n}$, where $m$ strands form the fixed
\textit{surgery braid} $B$ and $n$ strands form the {\it moving subbraid\/} $\beta$ representing the link $K$ in $M$. For an illustration see the middle picture of Figure~\ref{mixedbraidsLmoves}.
We further need the notions of the $L$-moves and the braid band moves.

\begin{defn} [$L$-moves and $\mathbb{Z}$-braid band moves, Definitions 2.1 and 5.6 \cite{LR1}] \rm

\noindent {\it(i)} Let $B \bigcup \beta$ be a geometric mixed braid in $S^3$ and $P$ a point of an arc of the moving subbraid $\beta$, such that $P$ is not vertically
aligned with any  crossing or endpoint of a braid strand. Doing an {\it $L$-move\/}  at $P$ means to cut the arc at $P$, to bend the two resulting
smaller arcs slightly apart  by a  small isotopy and to stretch them vertically, the upper downward and the lower upward, and both {\it over\/}
or {\it under\/} all other arcs of the diagram, so as to introduce two new corresponding moving strands with endpoints on the vertical line of the point $P$.
Stretching the new strands over  will give rise to an {\it $L_o$-move\/} and under to an {\it  $L_u$-move\/}. For an illustration see Figure~\ref{mixedbraidsLmoves}. Two geometric mixed braids shall be called {\it $L$-equivalent} if and only if they differ by a sequence of $L$-moves and braid isotopy.

\noindent {\it(ii)} A {\it geometric $\mathbb{Z}$-braid band move\/} is a move between geometric mixed braids which is a band
move between their closures. It starts with a little band oriented downward, which, before sliding along
a surgery strand, gets one twist {\it positive\/} or {\it negative\/} (see Figure~\ref{pbbm} (a) and (b)).
\end{defn}

\begin{figure}
\begin{center}
\includegraphics[width=5in]{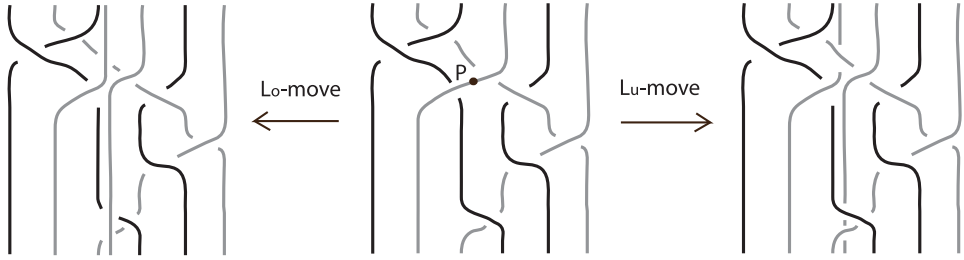}
\end{center}
\caption{A geometric mixed braid and the two types of $L$-moves. }
\label{mixedbraidsLmoves}
\end{figure}

\begin{figure}
\begin{center}
\includegraphics[width=5.4in]{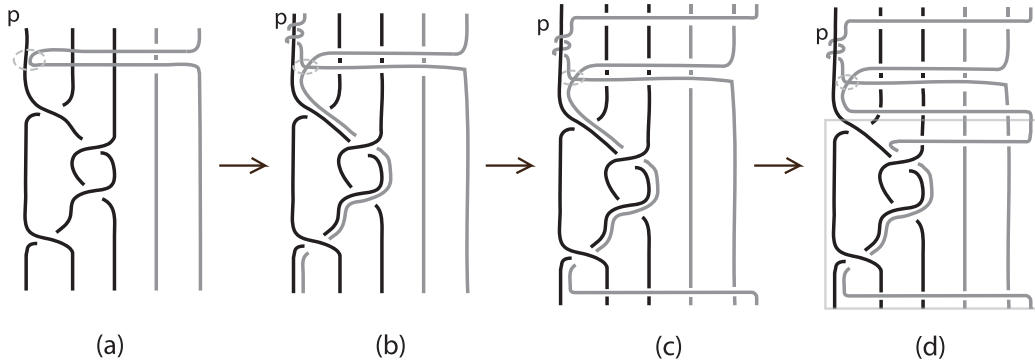}
\end{center}
\caption{  A geometric, a parted and an algebraic $\mathbb{Z}$-braid band move (top part of (d)). }
\label{pbbm}
\end{figure}

\begin{remark} \label{Lmove} \rm
{\it (i)} In \cite{LR1} it is shown that classical braid equivalence in $S^3$ is generated only by the $L$-moves. This implies that braid conjugation and in particular change of the order of the endpoints of a braid can be realized by $L$-moves. A demonstration can be found in \cite{LR2} Figure~14. \\
\noindent {\it (ii)} A geometric $\mathbb{Z}$-braid band move may be always assumed, up to $L$-equivalence, to take place at the top part of a mixed braid and on the right of the specific surgery strand (\cite{LR2} Lemma~5).
\end{remark}

In [LR1] the following theorem was proved for isotopic links in $M=\chi_{_{\mathbb{Z}}}(S^3, \widehat{B})$.

\begin{thm}[Geometric braid equivalence for $M=\chi_{_{\mathbb{Z}}}(S^3, \widehat{B})$, Theorem 5.10 \cite{LR1}] \label{geommarkov}
Two oriented links in $M=\chi_{_{\mathbb{Z}}}(S^3, \widehat{B})$ are isotopic if and only if any two corresponding geometric mixed braids in $S^3$ differ by
mixed braid isotopy, by $L$-moves that do not touch the fixed subbraid $B$ and by the geometric $\mathbb{Z}$-braid band moves.
\end{thm}

\subsection{Algebraic mixed braids and their equivalence}

Let $M=\chi_{_{\mathbb{Z}}}(S^3, \widehat{B})$. We will pass from the geometric braid equivalence to an algebraic statement for links in $M$.
An \textit{algebraic mixed braid} is a mixed braid on $m+n$ strands such that the first $m$ strands are fixed and form the identity braid on $m$ strands and the next $n$ strands are moving strands and represent a link in the manifold $M$. The set of all algebraic mixed braids on $m+n$ strands forms a subgroup of $B_{m+n}$, denoted $B_{m,n}$, and called {\it mixed braid group}. In [La2] the mixed braid groups $B_{m,n}$ have been introduced and studied and it is shown that $B_{m,n}$ has the presentation:

\begin{equation} \label{B}
B_{m,n} = \left< \begin{array}{ll}  \begin{array}{l}
a_1, \ldots, a_m,  \\
\sigma_1, \ldots ,\sigma_{n-1}  \\
\end{array} &
\left|
\begin{array}{l} \sigma_k \sigma_j=\sigma_j \sigma_k, \ \ |k-j|>1   \\
\sigma_k \sigma_{k+1} \sigma_k = \sigma_{k+1} \sigma_k \sigma_{k+1}, \ \  1 \leq k \leq n-1  \\
{a_i} \sigma_k = \sigma_k {a_i}, \ \ k \geq 2, \   1 \leq i \leq m    \\
 {a_i} \sigma_1 {a_i} \sigma_1 = \sigma_1 {a_i} \sigma_1 {a_i}, \ \ 1 \leq i \leq m  \\
 {a_i} (\sigma_1 {a_r} {\sigma^{-1}_1}) =  (\sigma_1 {a_r} {\sigma^{-1}_1})  {a_i}, \ \ r < i
\end{array} \right.  \end{array} \right>,
\end{equation}

\bigbreak

\noindent where the {\it loop generators} $a_i$ and the {\it braiding generators} $\sigma_j$ are as illustrated in Figure~\ref{Loops and Crossings}.

\begin{figure}
\begin{center}
\includegraphics[width=5.5in]{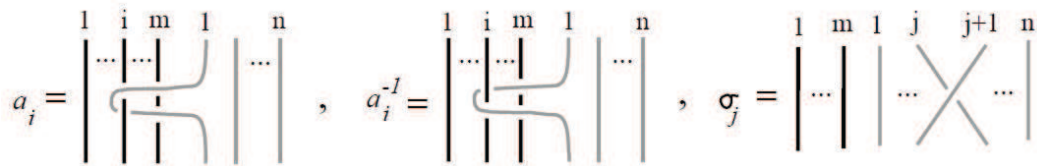}
\end{center}
\caption{ The loop generators $a_i, \ {a^{-1}_i}$ and the braiding generators $\sigma_j$ of $B_{m,n}$. }
\label{Loops and Crossings}
\end{figure}

\begin{figure}
\begin{center}
\includegraphics[width=6.2in]{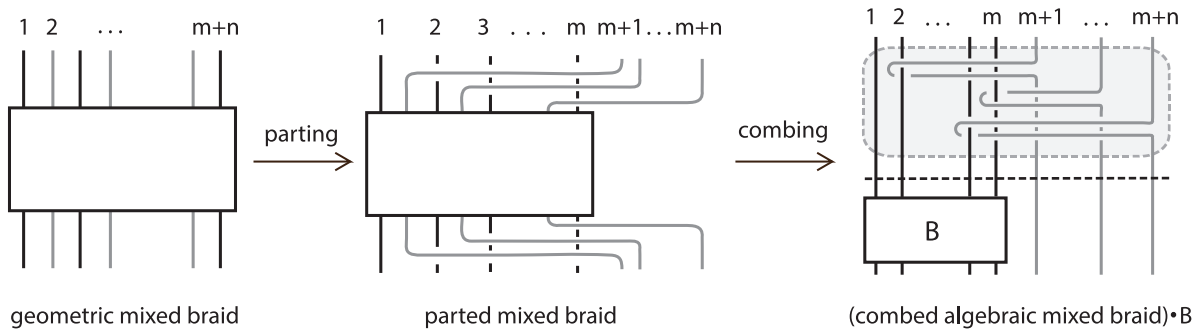}
\end{center}
\caption{ Parting and combing a geometric mixed braid. }
\label{partingandcombing}
\end{figure}

In order to give an algebraic statement for braid equivalence in $M$, we
first part the mixed braids and we translate the geometric $L$-equivalence of Theorem~\ref{geommarkov} to an equivalence of parted mixed braids. {\it Parting} a geometric mixed braid $B \bigcup\beta$ on $m+n$ strands means to separate its endpoints into two different sets, the first $m$ belonging to the subbraid $B$ and the last $n$ to $\beta$, and so that the resulting braids have isotopic closures. This is realized by pulling each pair of corresponding moving strands to the right and {\it over\/} or {\it under\/} each  strand of $B$ that lies on their right. We start from the rightmost pair respecting the position of the
endpoints. This process is called {\it parting} of a geometric mixed braid and the result is a {\it parted mixed braid}. If the strands are pulled always over the strands of $B$, then this parting is called {\it standard parting}. See the middle illustration of Figure~\ref{partingandcombing} for the standard parting of an abstract mixed braid. For more details the reader is referred to \cite{LR2}.

Then, in order to restrict Theorem~\ref{geommarkov} to the set of all parted mixed braids related to the manifold $M$, we need the following moves between parted mixed braids. {\it Loop conjugation} of a parted mixed braid $\beta$ is its  concatenation by a loop $a_i$ (or by ${a^{-1}_i}$) from above and from ${a^{-1}_i}$ (corr. $a_i$) from below, that is $\beta \sim a^{\pm1} \beta a^{\mp1}$. As it turns out, two partings of a geometric mixed braid differ by loop conjugations (cf. Lemma~2 \cite{LR2}). A {\it parted $L$-move} is an $L$-move between parted mixed braids. Further, a mixed braid with an $L$-move performed can be parted to a parted mixed braid with a parted $L$-move performed. Namely we make the parting consistent with the label of the $L$-move: an $L_o$ move will be parted by pulling over all other strands, while an $L_u$ move will be parted by pulling under all other strands (cf. Lemma~3 \cite{LR2}).
A {\it parted $\mathbb{Z}$-braid band move} is a geometric $\mathbb{Z}$-braid band move between parted mixed braids, such that it takes place at the top part of the braid and the little band {\it starts from the last strand} of the moving subbraid and it {\it moves over\/} each moving strand and each component of the surgery braid, until it reaches from the right the specific component, and then is followed by parting (see Figure~\ref{pbbm}(c)). Moreover, performing a  $\mathbb{Z}$-braid band move on a mixed braid and then parting, the result is equivalent, up to $L$-moves and loop conjugation, to performing a parted $\mathbb{Z}$-braid band move (cf. Lemma~5 \cite{LR2}).

\begin{thm}[Parted mixed braid equivalence for $M=\chi_{_{\mathbb{Z}}} (S^3, \widehat B)$, Theorem~3 \cite{LR2}]\label{partedcco} Two
oriented links in  $M=\chi_{_{\mathbb{Z}}} (S^3, \widehat B)$ are isotopic if and only if any two
corresponding  parted mixed braids differ by a finite sequence of parted mixed braid isotopies, parted
$L$-moves, loop  conjugations and parted $\mathbb{Z}$-braid band moves.
\end{thm}

We now comb the parted mixed braids in order to translate the parted mixed braid equivalence to an equivalence between algebraic mixed braids.  {\it Combing} a parted mixed braid means to separate the knotting and linking of the moving part away from the fixed subbraid using mixed braid isotopy. More precisely, let $\Sigma_k$ denote the crossing between the $k^{th}$ and the $(k+1)^{st}$ strand of the fixed subbraid. Then, for all $j=1,\ldots,n-1$ and $k=1,\ldots,m-1$ we have: $\Sigma_k \sigma_j = \sigma_j \Sigma_k$.
Thus, the only generating elements of the moving part that are affected by the combing are the loops $a_i$. This is illustrated in Figure~\ref{combing}.  In Lemma~6 \cite{LR2} formuli are given for the effect of combing on the $a_i$'s (see Lemma~\ref{combinglem} below).

\begin{figure}
\begin{center}
\includegraphics[width=5.2in]{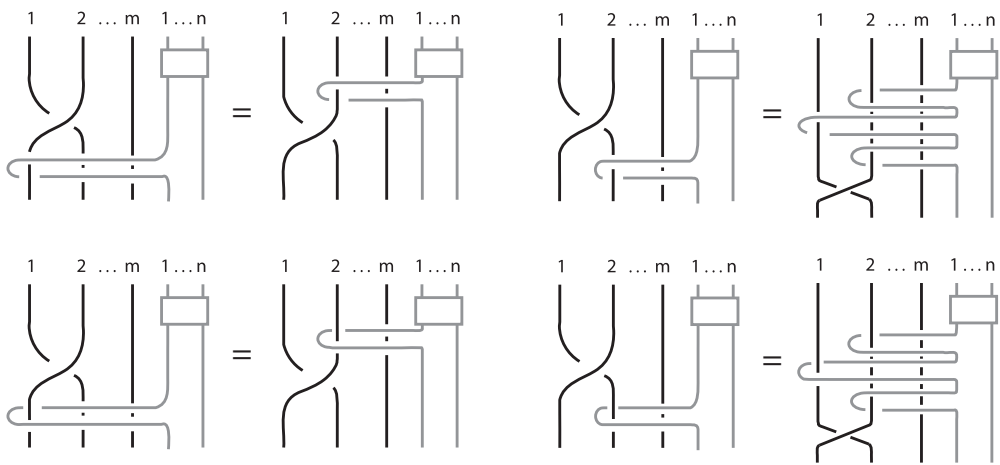}
\end{center}
\caption{ Combing a parted mixed braid. }
\label{combing}
\end{figure}

The effect of combing a parted mixed braid is to separate it into two distinct parts:
the {\it algebraic part} at the top, which has all fixed strands forming the identity braid, so is an element of some mixed braid group $B_{m,n}$, and which contains all the knotting and linking information of the link $L$ in $M$;
the {\it coset part} at the bottom, which contains only the fixed subbraid $B$ and an identity braid for the moving part (see right hand most illustration in Figure~\ref{partingandcombing}). Let now $C_{m,n}$ denote the set of parted mixed braids on $n$ moving strands with fixed subbraid $B$. Concatenating two elements of $C_{m,n}$ is not a closed operation since it alters the braid description of the manifold. However, as a result of the combing, for the fixed subbraid $B$ the set $C_{m,n}$ is a coset of $B_{m,n}$ in $B_{m+n}$. Fore details on the above the reader is referred to \cite{La2}.

Translating the parted braid equivalence into an equivalence between algebraic mixed braids, we will obtain an algebraic statement of Theorem~\ref{partedcco}. Since loop conjugation does not take into account the combing of the loop through the fixed subbraid, we need the notion of {\it combed loop conjugation}. A combed loop conjugation is a move between algebraic mixed braids and is the result of a loop conjugation on a combed mixed braid followed by combing, so it can be described algebraically as: $\beta \sim \alpha_i^{\mp1} \beta \rho_i^{\pm1}$ for $\beta, a_i, \rho_i \in B_{m,n}$, where $\rho_i$ is the combing of the loop  $a_i$ through the fixed subbraid $B$. We also define {\it algebraic $M$-conjugation} of an algebraic mixed braid to be its conjugation by a crossing $\sigma_j$ (or by ${\sigma^{-1}_j}$). An {\it algebraic $M$-move} is defined to be the insertion of a crossing ${\sigma^{\pm 1}_n}$ on the right hand side of an algebraic mixed braid. Finally, an {\it algebraic $L$-move} is defined to be a $L$-move between algebraic mixed braids. An algebraic $L$-move has the following algebraic expression for an $L_o$-move and an $L_u$-move respectively:

\begin{equation}
\begin{array}{lll}
\alpha &=& \alpha_1\alpha_2 \stackrel{L_o}{\sim}
\sigma_i^{-1}\ldots \sigma_n^{-1} \alpha_1' \sigma_{i-1}^{-1}\ldots
\sigma_{n-1}^{-1}\sigma_n^{\pm 1}\sigma_{n-1} \ldots \sigma_i
\alpha_2' \sigma_n \ldots \sigma_i \\
\alpha&=&\alpha_1\alpha_2 \stackrel{L_u}{\sim}
\sigma_i\ldots \sigma_n \alpha_1' \sigma_{i-1}\ldots
\sigma_{n-1}\sigma_n^{\pm 1}\sigma_{n-1}^{-1}\ldots\sigma_i^{-1}
\alpha_2' \sigma_n^{-1}\ldots\sigma_i^{-1}
\end{array}
\end{equation}

\noindent  where $\alpha_1$,
 $\alpha_2$ are elements of $B_{m,n}$ and $\alpha_1'$,
 $\alpha_2' \in B_{m,n+1}$ are obtained from $\alpha_1$,
 $\alpha_2$ by replacing each $\sigma_j$ by $\sigma_{j+1}$ for
 $j=i,\ldots,n-1$.

Note that algebraic $M$-conjugation, the algebraic $M$-moves and the algebraic $L$-moves commute with combing. Note also that Remark~\ref{Lmove}(i) applies equally to the case of algebraic mixed braids (cf. Lemma~4 \cite{LR2}).

We finally need to understand how a parted $\mathbb{Z}$-braid band move is combed through the surgery braid $B$.

\begin{defn}[Definition~7 \cite{LR2}] \label{twistband} \rm  An  {\it algebraic  $\mathbb{Z}$-braid band move\/} is defined to be a parted band move
between algebraic mixed braids (see top part of Figure~\ref{pbbm}(d)). Setting:
\[
\lambda_{n-1,1} := \sigma_{n-1} \ldots \sigma_1  \mbox{ \ \ and \ \ }
t_{k,n} := \sigma_n \ldots \sigma_1 a_k {\sigma^{-1}_1} \ldots {\sigma^{-1}_n},
\]
an algebraic band move has the following algebraic expression:
\[
\beta_1 \beta_2 \ \sim \ \beta^{\prime}_1 \, {t^{p_k}_{k,n}} \,
{\sigma^{\pm 1}_n} \, \beta^{\prime}_2,
\]
\noindent where $\beta_1, \beta_2 \in B_{m,n}$ and $\beta^{\prime}_1, \beta^{\prime}_2  \in
B_{m,n+1}$  are the words $\beta_1, \beta_2$ respectively with the substitutions:
\[
\begin{array}{lcl}
  {a^{\pm 1}_k} & \longleftrightarrow & {[({\lambda^{-1}_{n-1,1}}
{\sigma^{2}_n} \lambda_{n-1,1}) \, a_k]}^{\pm 1}    \\

  {a^{\pm 1}_i} & \longleftrightarrow &   ({\lambda^{-1}_{n-1,1}} {\sigma^{2}_n}
 \lambda_{n-1,1}) \, {a^{\pm 1}_i} \, ({\lambda^{-1}_{n-1,1}} {\sigma^{2}_n}
\lambda^{-1}_{n-1,1}),   \mbox{ \ \ if \ } i < k   \\

 {a^{\pm 1}_i}  & \longleftrightarrow & {a^{\pm 1}_i},  \mbox{ \ \ if \ } i > k.    \\
 \end{array}
\]

Further, a {\it combed algebraic $\mathbb{Z}$-braid band move\/} is a move between algebraic mixed braids and is defined to be a parted $\mathbb{Z}$-braid band move that has been combed through $B$. So it is the composition of an algebraic $\mathbb{Z}$-braid band move with the combing of the parallel strand and it has the following algebraic expression:
\[
\beta_1 \beta_2 \ \sim \ \beta^{\prime}_1 \, {t^{p_k}_{k,n}} \,
{\sigma^{\pm 1}_n} \, \beta^{\prime}_2 \, r_k,
\]
\noindent where $r_k$ is the combing
 of the parted parallel strand to the $k^{th}$ surgery strand through the surgery braid.
 \end{defn}

The group $B_{m,n}$ embeds naturally into the group $B_{m,n+1}$. We shall denote $B_{m,\infty}=\bigcup_{n=1}^{\infty }B_{m,n}$
and similarly $C_{m,\infty}=\bigcup_{n=1}^{\infty }C_{m,n}$.

We are now in position to give the algebraic Markov theorem for $M=\chi_{_{\mathbb{Z}}} (S^3, \widehat B)$.

\begin{thm}[Algebraic Markov Theorem for $M=\chi_{_{\mathbb{Z}}} (S^3, \widehat B)$, Theorem~5 \cite{LR2}] \label{algcco} \ Two  oriented links
in  $M=\chi_{_{\mathbb{Z}}} (S^3, \widehat B)$ are isotopic if and only if any two
corresponding algebraic mixed braid representatives in  $B_{m,\infty}$  differ by a finite
sequence of the following moves:

\vspace{.03in}
\noindent (1) \ Algebraic $M$-moves: \ ${\beta}_1 {\beta}_2 \sim
{\beta}_1{\sigma^{\pm 1}_n}{\beta}_2, \ \ for \ {\beta}_1,
{\beta}_2  \in B_{m,n}$,

\noindent (2) \ Algebraic $M$-conjugation: \
$\beta \sim {\sigma^{\pm 1}_j} \beta {\sigma^{\mp 1}_j}$, \ \ for \ $\beta, \sigma_j \in B_{m,n}$,

\vspace{.03in}
\noindent (3) \ Combed loop conjugation: \
$\beta  \sim  {a^{\mp 1}_i} \beta {\rho^{\pm 1}_i}$, \ \  for \ $\beta, a_i, \rho_i \in B_{m,n}$, where
$\rho_i$ is the combing of the loop  $a_i$ through $B$,

\vspace{.03in}
\noindent (4) \  Combed algebraic braid band moves: \ For  for every $k=1,\ldots,m$ we have:
\[ \beta_1 \beta_2 \ \sim \ \beta^{\prime}_1 \, {t^{p_k}_{k,n}} \,
{\sigma^{\pm 1}_n} \, \beta^{\prime}_2 \, r_k,
\]
\noindent where  $\beta_1, \beta_2 \in B_{m,n}$ and  $\beta^{\prime}_1, \beta^{\prime}_2  \in B_{m,n+1}$ are as in
Definition~\ref{twistband} and where $r_k$  is the  combing of the parted parallel strand  to the $k$th surgery strand through $B$.

\smallbreak

Equivalently, by a finite sequence of algebraic mixed braid relations and the following moves:

\vspace{.03in}
\noindent (1$'$) \,  algebraic $L$-moves,

\noindent (2$'$) \, combed loop conjugations,

\noindent (3$'$) \ combed algebraic $\mathbb{Z}$-braid band moves.
\end{thm}

\section{The Reidemeister Theorem for links in  $3$-manifolds with rational surgery description}

Let $M$ be a c.c.o. $3$-manifold obtained from $S^3$ by rational surgery, that is surgery along a framed link $\widehat{B}$ with rational coefficients. We shall denote  $M=\chi_{_{\mathbb{Q}}}(S^3, \widehat{B})$. Let also $L$ be an oriented link in $M$. By the discussion in Section~1.1, isotopy in $M$ is translated into isotopy in $S^3\backslash \widehat{B}$ together with the two types, $\alpha$ and $\beta$, of band moves for mixed links in $S^3$.
The band moves in this case are described as follows.
Let $c$ be a component of $\widehat{B}$ with framing $p/q$. The specified parallel curve $l$ of $c$ is a $(p,q)$-torus knot on the boundary of a tubular neighborhood of $c$ which, by construction, bounds a disc in $M$. Then, a {\it $\mathbb{Q}$-band move} along $c$ is the connected sum of a component of $L$ with the $(p,q)$-torus knot $l$ and there are two types, $\alpha$ and $\beta$, according to the orientations. The two types of band moves are illustrated in Figure~\ref{bmoves}, where $c$ is a trefoil knot with $2/3$ surgery coefficient and where ``band move" is shortened to ``b.m.".

\begin{figure}
\begin{center}
\includegraphics[width=3.6in]{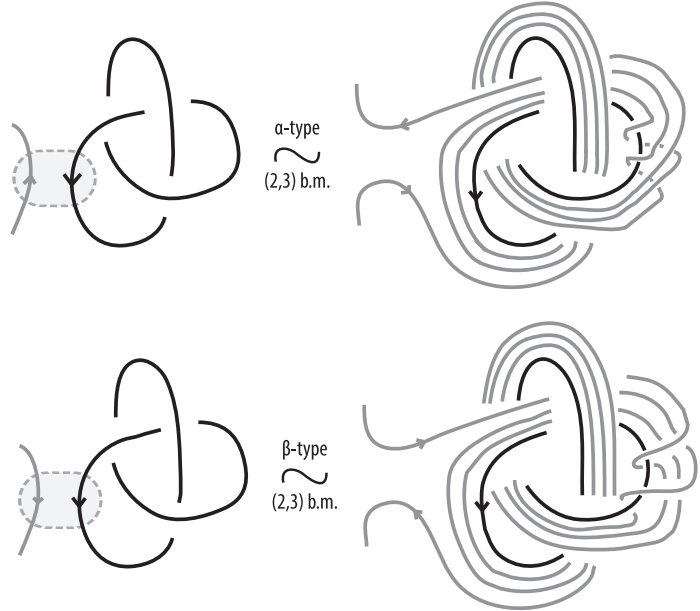}
\end{center}
\caption{ The two types of $\mathbb{Q}$-band moves. }
\label{bmoves}
\end{figure}

Clearly, Theorem~\ref{reidemintegral} applies also to $M=\chi_{_{\mathbb{Q}}}(S^3, \widehat{B})$. Namely we have:

\begin{thm}[Reidemeister for  $M=\chi_{_{\mathbb{Q}}}(S^3, \widehat{B})$ with two types of band moves] \label{reidemrationalboth}
Two oriented links $L_1$, $L_2$ in  $M=\chi_{_{\mathbb{Q}}}(S^3, \widehat{B})$ are isotopic if and only if any two corresponding mixed
link diagrams of theirs, $\widehat{B} \cup \widetilde{L_1}$ and $\widehat{B} \cup \widetilde{L_2}$, differ by isotopy in $S^3 \backslash \widehat{B}$
together with a finite sequence of the two types $\alpha$ and $\beta$ of band moves.
\end{thm}

In this section we sharpen Theorem~\ref{reidemrationalboth}. More precisely, we show that only one of the two types of band moves is necessary in order to describe isotopy for links in $M$.
The proof is based on a known contrivance, which was used in the proof of Theorem~5.10 \cite{LR1} (recall Theorem~\ref{geommarkov}) for establishing the sufficiency of the geometric braid band moves in the mixed braid equivalence where the case of integral surgery only is considered (see Figure~\ref{Unknotted surgery}). Theorem~\ref{reidemrational} simplifies the proof of Theorem~\ref{geommarkovrat}. 

\begin{thm} [Reidemeister for  $M=\chi_{_{\mathbb{Q}}}(S^3, \widehat{B})$ with one type of band moves] \label{reidemrational}
Two oriented links $L_1$, $L_2$ in  $M=\chi_{_{\mathbb{Q}}}(S^3, \widehat{B})$ are isotopic if and only if any two corresponding mixed
link diagrams of theirs, $\widehat{B}\bigcup L_1$ and $\widehat{B}\bigcup L_2$, differ by a finite sequence of the band moves of type $\alpha$ (or equivalently of type $\beta$) and isotopy in $S^3\backslash \widehat{B}$.
\end{thm}

\begin{proof}
Let $L$ be an oriented link in $M$. By Theorem~\ref{reidemrationalboth}, it suffices to show that a band move of type $\beta$ can be obtained from a band move of type $\alpha$ and isotopy in the knot complement. We will first demonstrate the proof for an unknotted surgery component $c$ with integral coefficient $p$. (Note that integral surgery description can be considered as a special case of rational surgery description.) We shall follow the steps of the proof in Figure~\ref{Unknotted surgery} where $p=2$. We start with performing a band move of type $\beta$ using a component $s$ of the link $L$. In Figure~\ref{Unknotted surgery}(b) we see the two twists of the band move wrapping around the surgery curve $c$ in the righthand sense. Then, using an arc of the same link component $s$, we perform a second band move of type $\alpha$. This will take place within a thinner tubular neighborhood than
the first band move. So, the two twists of the second band move, which also wrap around $c$ in the righthand sense, commute with the two twists of the first band move (see Figure~\ref{Unknotted surgery}(c)). We arrange all $2p$ twists in pairs as follows. We pass one twist from the second band move (the closest) through all twists of the first band move, see Figure~\ref{Unknotted surgery}(f). Since all twists follow the righthand sense, the two innermost twists coming from the second and the first band move, create a little band (Figure~\ref{Unknotted surgery}(f)) which can be eliminated using isotopy in the knot complement of $c$ (Figure~\ref{Unknotted surgery}(g) and Figure~\ref{Twist Cancelation}). This is the cancelation of the first pair of the $2p$ twists. Repeating the same procedure we cancel all $p$ pairs and we end up with the component $s$ of the link $L$ as it was in the initial position before the band moves (see Figures~\ref{Unknotted surgery}(i), (a)).
 
\begin{figure}
\begin{center}
\includegraphics[width=5.9in]{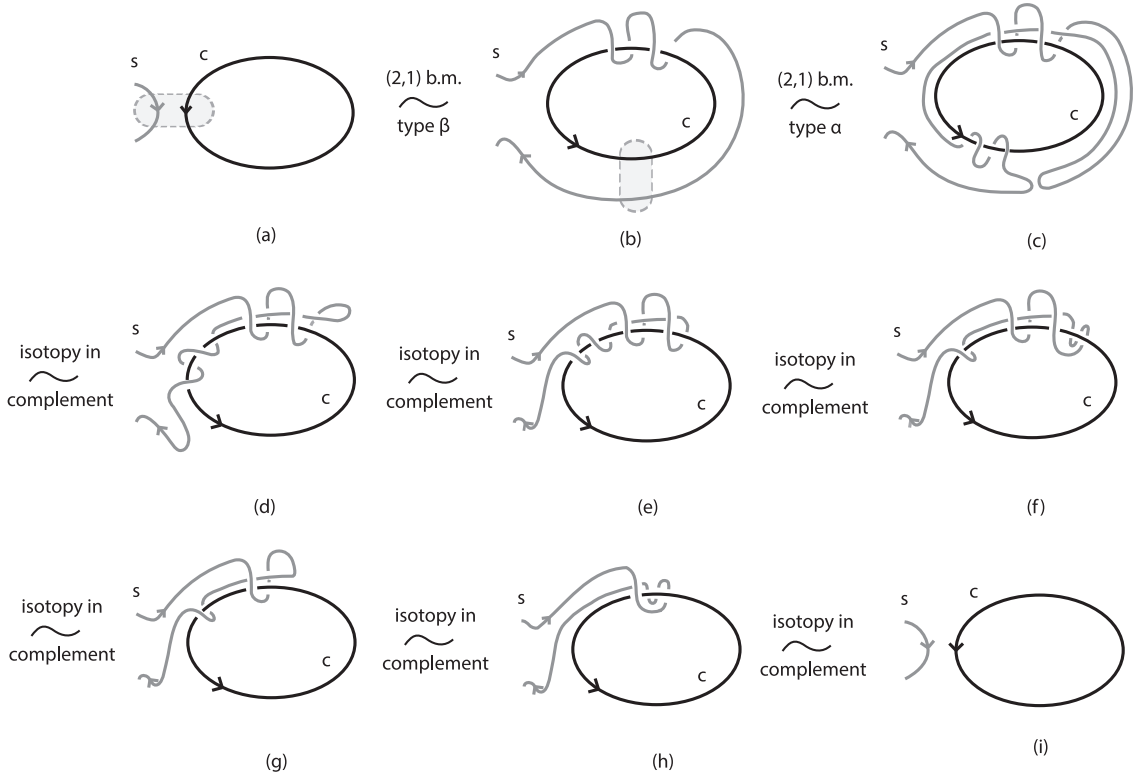}
\end{center}
\caption{A type-$\beta$ band move follows from a type-$\alpha$ band move in the case of integral surgery coefficient. }
\label{Unknotted surgery}
\end{figure}

\begin{figure}
\begin{center}
\includegraphics[width=3.8in]{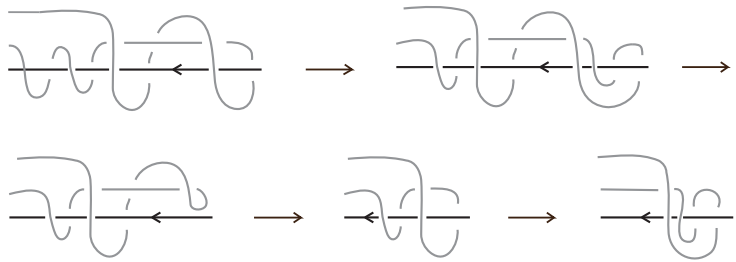}
\end{center}
\caption{ Twist cancelation.}
\label{Twist Cancelation}
\end{figure}

For the more general case of rational surgery along any knot $c$ we follow the same idea. More precisely, we perform a $\mathbb{Q}$-band move of type $\beta$ along $c$ and we obtain an outer $(p,q)$-torus knot. Then, we perform a $\mathbb{Q}$-band move of type $\alpha$ along $c$ and we obtain an inner $(p,q)$-torus knot. In Figure~\ref{Reidemeister Theorem1} we illustrate this for the case where $p=2$, $q=3$ and $c$ a trefoil knot.

\begin{figure}
\begin{center}
\includegraphics[width=6in]{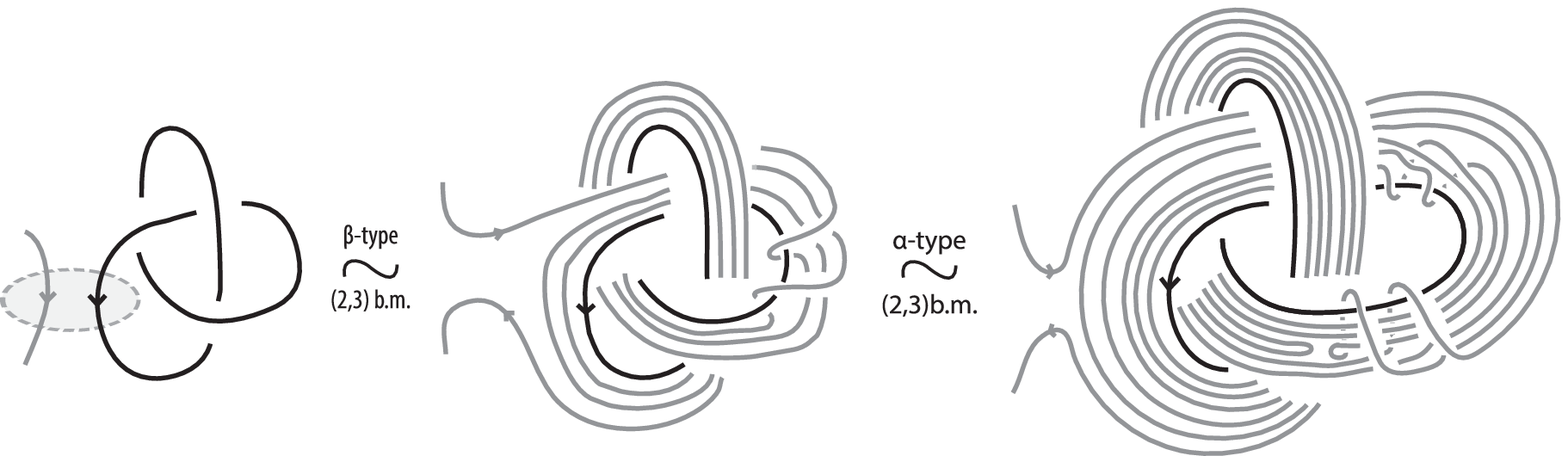}
\end{center}
\caption{ A band move of type $\beta$ followed by a band move of type $\alpha$. }
\label{Reidemeister Theorem1}
\end{figure}

Without loss of generality (by isotopy in the complement of $c$), the second band move is performed on the innermost arc of the $q$ arcs parallel to $c$, creating $q$ new parallel arcs even closer to $c$. After the second $\mathbb{Q}$-band move is performed, the outer arc of the $q$ new arcs and the inner arc of the $q$ arcs coming from the first band move of type $\alpha$ form a band (see shaded area in Figure~\ref{Reidemeister Theorem2}). Then, using isotopy in the complement of $c$, we eliminate this band by pulling it along $c$. This will result in the elimination of $p-q$ pairs of parallel arcs to $c$.
In our example, this is done in two steps (see Figures~\ref{Reidemeister Theorem2} and \ref{Reidemeister Theorem3}).

\begin{figure}
\begin{center}
\includegraphics[width=5.1in]{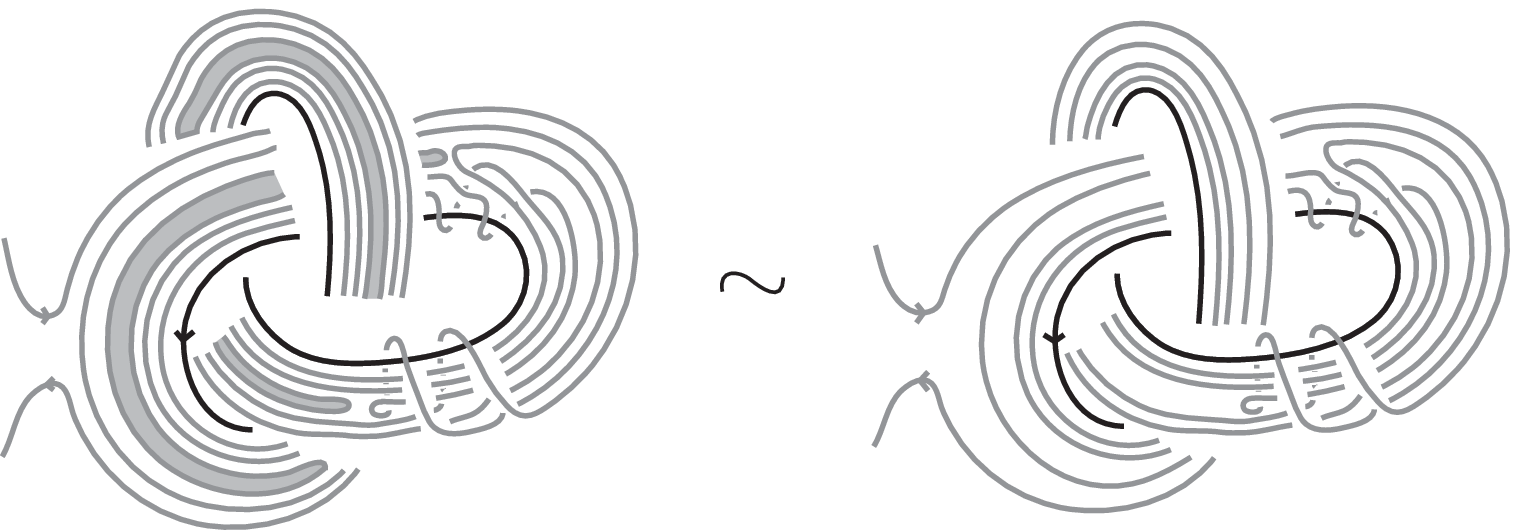}
\end{center}
\caption{ Band with boundary two parallel arcs of opposite orientations. }
\label{Reidemeister Theorem2}
\end{figure}

\begin{figure}
\begin{center}
\includegraphics[width=5.1in]{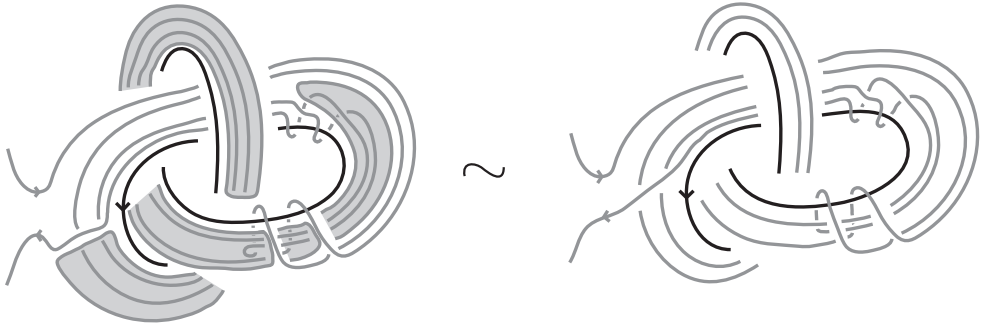}
\end{center}
\caption{ Retracting the band along the surgery component. }
\label{Reidemeister Theorem3}
\end{figure}

As in the case of integral surgery the twists coming from the two band moves commute (see Figure~\ref{Reidemeister Theorem4}). Arranging these $2p$ twists pairwise (Figure~\ref{Reidemeister Theorem5}), they cancel out by the fact that all twists have the same handiness, but opposite orientation. In the end, $s$ is left as in its initial position.

\begin{figure}
\begin{center}
\includegraphics[width=5.1in]{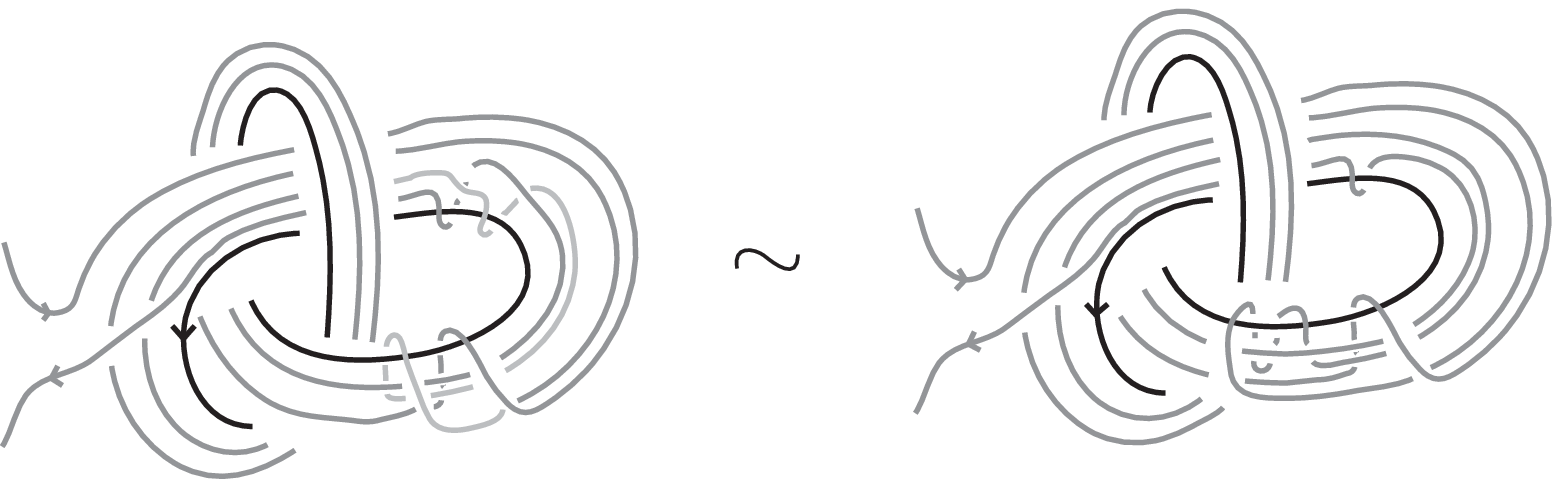}
\end{center}
\caption{ Arranging twists in pairs to form bands. }
\label{Reidemeister Theorem4}
\end{figure}

\begin{figure}
\begin{center}
\includegraphics[width=4.7in]{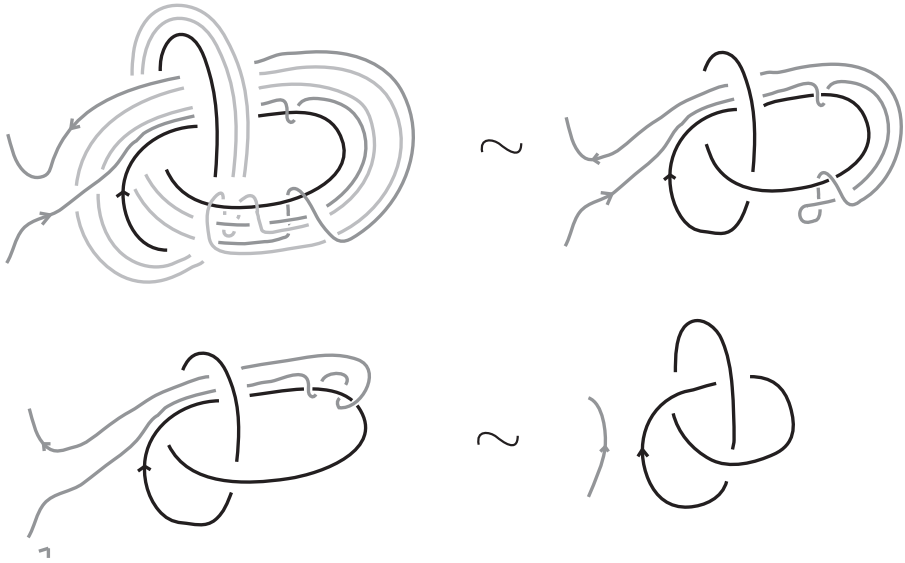}
\end{center}
\caption{ Twist cancelation. }
\label{Reidemeister Theorem5}
\end{figure}

So, a $\mathbb{Q}$-band move of type $\beta$ can be performed using a $\mathbb{Q}$-band move of type $\alpha$ and isotopy in the complement of the surgery component $c$.
The proof of Theorem~\ref{reidemrational} is now concluded.
\end{proof}

\section{Geometric braid equivalence for links in a $3$-manifold with rational surgery description}

In this section we extend Theorem~\ref{geommarkov} to manifolds with rational surgery description, that is $M=\chi_{_{\mathbb{Q}}}(S^3, \widehat{B})$, taking into account the sharpened Reidemeister theorem for $M$ (Theorem~\ref{reidemrational}). We first need the following.

\begin{defn} \rm
A {\it geometric $\mathbb{Q}$-braid band move\/} is a move between geometric mixed braids which is a $\mathbb{Q}$-band
move of type $\alpha$ between their closures. It starts with a little band (an arc of the moving subbraid) close to a surgery strand with surgery coefficient $p/q$. The little band gets first one twist {\it positive\/} or {\it negative\/}, which shall be denoted as $c^{\prime}_{\pm}$ and then is replaced by $q$ strands that run in parallel to all strands of the same surgery component and link only with that surgery strand, wrapping around it $p$ times and, thus, forming a $(p,q)$-torus knot, Figure~\ref{lemma}. This braided $(p,q)$-torus knot is denoted as $d^{\prime}$. A geometric $\mathbb{Q}$-braid band move with a positive (resp. negative) twist shall be called a positive geometric $\mathbb{Q}$-braid band move (resp. negative geometric $\mathbb{Q}$-braid band move).
\end{defn}

\begin{figure}
\begin{center}
\includegraphics[width=3in]{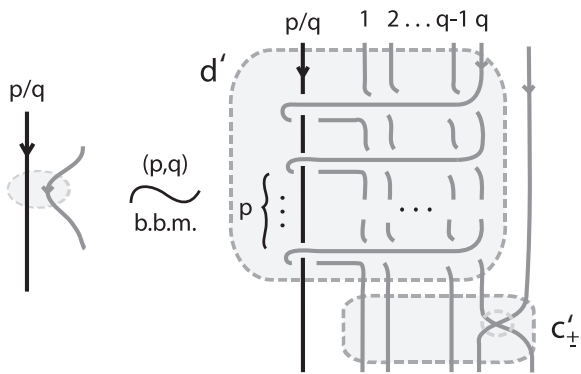}
\end{center}
\caption{ A $\mathbb{Q}$-braid band move.}
\label{lemma}
\end{figure}

By Remark~1(ii) a $\mathbb{Q}$-braid band move may be assumed to take place at the top part of a mixed braid and all strands from a $\mathbb{Q}$-braid band move may be assumed to lie on the righthand side of the surgery strands. We shall now prove the following.

\begin{thm} [Geometric braid equivalence for $M=\chi_{_{\mathbb{Q}}}(S^3, \widehat{B})$]  \label{geommarkovrat}
Two oriented links in $M=\chi_{_{\mathbb{Q}}}(S^3, \widehat{B})$ are isotopic if and only if any two corresponding geometric mixed braids in $S^3$ differ by
mixed braid isotopy, by $L$-moves that do not touch the fixed subbraid $B$ and by the geometric $\mathbb{Q}$-braid band moves.
\end{thm}

\begin{proof}
The proof is completely analogous to and is based on the proof of Theorem 5.10 \cite{LR1} (Theorem~\ref{geommarkov}).
Let $K_1$ and $K_2$ be two isotopic oriented links in $M$. By Theorem~\ref{reidemrational}, the corresponding mixed links $\widehat{B} \bigcup K_1$ and $\widehat{B} \bigcup K_2$ differ by isotopy in the complement of $\widehat{B}$ and $\mathbb{Q}$-band moves of type $\alpha$. Note that, by Theorem~\ref{reidemrational} we do not need to consider band moves of type $\beta$.
By Theorem 5.10 \cite{LR1}, isotopy in the complement of $\widehat{B}$ translates into geometric braid isotopy and the $L$-moves.
Let now  $\widehat{B} \bigcup K_1$ and $\widehat{B} \bigcup K_2$ differ by a $\mathbb{Q}$-band move of type $\alpha$ (recall Figure~\ref{bmoves}).
Let $\widehat{B} \bigcup \widetilde{K_1}$ and $\widehat{B} \bigcup \widetilde{K_2}$ be two mixed link diagrams of the mixed links $\widehat{B} \bigcup K_1$ and $\widehat{B} \bigcup K_2$ which differ only by the places illustrated in Figure~\ref{markovtheorem}. As in \cite{LR1}, by the braiding algorithm given therein, the diagrams $\widehat{B} \bigcup \widetilde{K_1}$ and $\widehat{B} \bigcup \widetilde{K_2}$ may be assumed braided everywhere except for the places where the $\mathbb{Q}$-band move is performed.

We now braid the up-arc in Figure~\ref{markovtheorem}(b) and obtain a geometric mixed braid $\widehat{B} \bigcup b_1$ corresponding to the diagram $\widehat{B} \bigcup \widetilde{K_1}$ (see Figure~\ref{markovtheorem}(a)). Note that Figure~\ref{markovtheorem}(c) is already in braided form and let $B \bigcup b_2$ denote
the geometric mixed braid corresponding to the diagram $\widehat{B} \bigcup \widetilde{K_2}$.

\begin{figure}
\begin{center}
\includegraphics[width=4.8in]{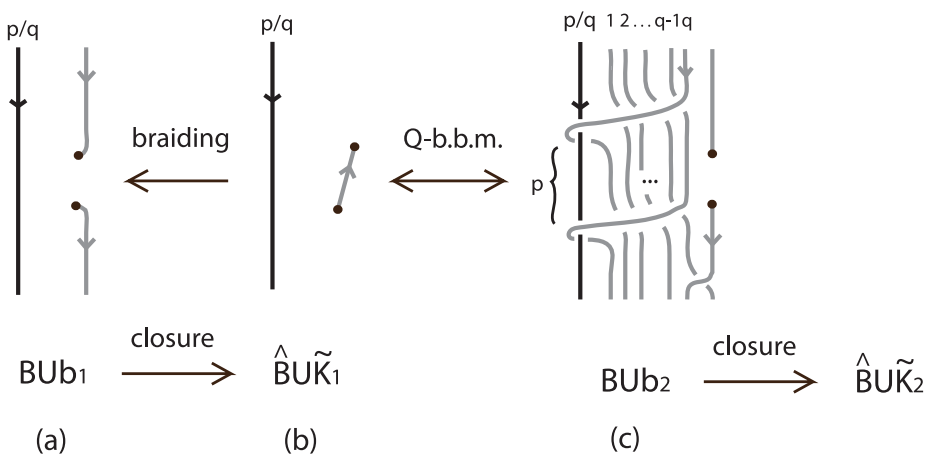}
\end{center}
\caption{ A type $\alpha$ band move (locally) and its braiding. }
\label{markovtheorem}
\end{figure}

We would like to show that the two mixed braids $B \bigcup b_1$ and $B \bigcup b_2$ differ by the moves given in the statement of the Theorem.

We perform a Reidemeister I move on $\widehat{B} \bigcup \widetilde{K_1}$ with a {\it negative} crossing and obtain the diagram $\widehat{B} \bigcup \widetilde{K'_1}$. Then, the corresponding mixed braids, $B \bigcup b_1$ and $B \bigcup b'_1$, differ by mixed braid isotopy and $L$-moves (see Figure~\ref{markovproof}(a) and (b)). We then perform a {\it positive} $\mathbb{Q}$-braid band move on $B \bigcup b'_1$ and obtain the mixed braid $B \bigcup b'_2$. In the closure of $B \bigcup b'_2$ we unbraid and re-introduce the two up-arcs illustrated in Figure~\ref{markovproof}(b), obtaining a diagram $\widehat{B} \bigcup \widetilde{K'_2}$ with the formation of a Reidemeister II move. Performing this move on $\widehat{B} \bigcup \widetilde{K'_2}$ we obtain the diagram $\widehat{B} \bigcup \widetilde{K_2}$, which is already in braided form and its corresponding mixed braid is $B \bigcup b_2$ (see Figure~\ref{markovproof}(c) and (d)). So, the mixed braids $B \bigcup b'_2$ and $B \bigcup b_2$ differ by mixed braid isotopy and $L$-moves.
Therefore, we showed that the braids $B \bigcup b_1$ and $B \bigcup b_2$ in Figure~\ref{markovtheorem}(a) and (c) differ by mixed braid isotopy, $L$-moves and a braid band move.
This concludes the proof.
\end{proof}

\begin{figure}
\begin{center}
\includegraphics[width=6.3in]{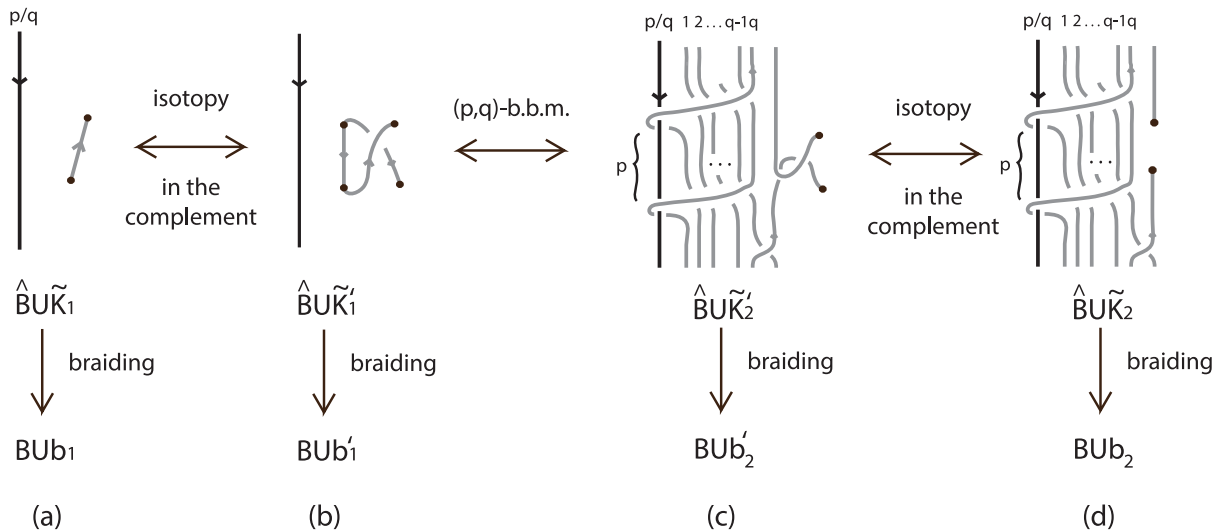}
\end{center}
\caption{ The steps of the proof of Theorem~\ref{geommarkovrat}. }
\label{markovproof}
\end{figure}

\section{Algebraic braid equivalence for links in a $3$-manifold with rational surgery description}

In order to translate the geometric mixed braid equivalence to an equivalence of algebraic mixed braids we follow the strategy in \cite{LR2} for integral surgery description. Namely, we apply first parting and then combing to the geometric mixed braids. What makes things more complicated in the case of rational surgery description is that the surgery braid $B$ is in general not a pure braid and when we apply a $\mathbb{Q}$-braid band move on a mixed braid, the little band that approaches the surgery strand is replaced by $q$ strands that run in parallel to all strands of the same surgery component.

In order to adopt and apply results from \cite{LR2} we need the notion of a {\it $\mathbb{Q}$-strand cable}.

\begin{defn} \label{cable}
\rm
We define a {\it $q$-strand cable} to be a set of $q$ parallel strands coming from a $\mathbb{Q}$-braid band move and following one strand of the specified surgery component.
\end{defn}

Let now $B\bigcup \beta$ be a geometric mixed braid and suppose that a $\mathbb{Q}$-braid band move is performed on $B\bigcup \beta$. Treating the new strands coming from the braid band move as cables running in parallel to the strands of a surgery component, that is, treating each cable as one thickened strand, we may part the geometric mixed braid following the exact procedure as in \cite{LR2}. More precisely, we have the following.

\begin{lemma} \label{partcable}
Cabling and standard parting commute. That is, standard parting of a mixed braid with a $\mathbb{Q}$-braid band move performed and then cabling, is the same as cabling first the set of new strands and then standard parting.
\end{lemma}

\begin{proof}
Let $B\bigcup \beta$ be a geometric mixed braid on $m+n$ strands and let a $\mathbb{Q}$-braid band move be performed on a surgery component $s$ of $B$. Let also $s_1, \ldots , s_k \in \{ 1, \ldots, m \}$ be the numbers of the strands of the surgery component $s$ and let $c_1, \ldots , c_k$ denote the $q$-strand cables corresponding to $s_1, \ldots , s_k$. On the one hand, after the $\mathbb{Q}$-braid band move is performed and before any cablings occur, we part the geometric mixed braid following the procedure of the standard parting as described in Section~1.3 (see Figure~\ref{stndpartcable}(a)). On the other hand we cable first each set of $q$-strands resulting from the $\mathbb{Q}$-braid band move and then we part the geometric mixed braid with the standard parting, treating each cable as one (thickened) strand, see Figure~\ref{stndpartcable}(b). Since both cabling and parting a geometric mixed braid respect the position of the endpoints of each pair of corresponding moving strands, it follows that cabling and parting commutes.
\end{proof}

\begin{figure}
\begin{center}
\includegraphics[width=6.2in]{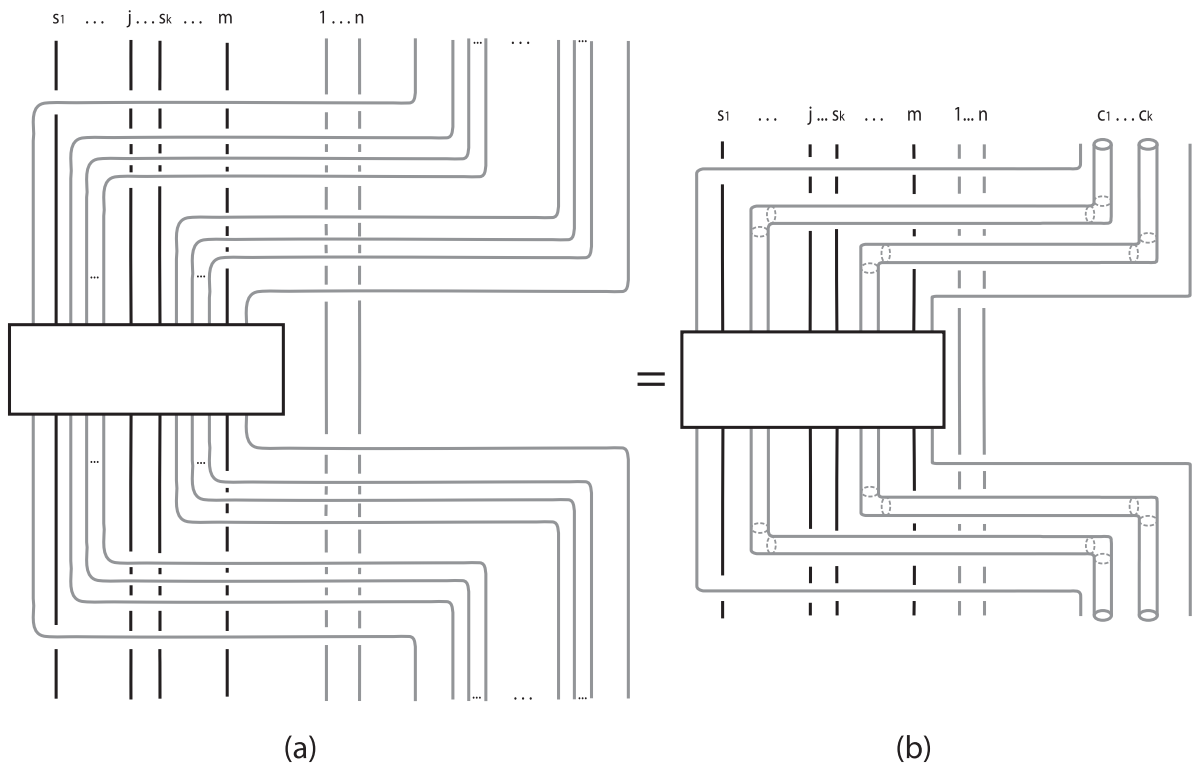}
\end{center}
\caption{ Standard parting and cabling.}
\label{stndpartcable}
\end{figure}

Recall from Section~1.3 that a geometric $L$-move can be turned to a parted $L$-move. In order to give the analogue of Theorem~\ref{partedcco} in the case of rational surgery we also need to introduce the following adaptation of a parted $\mathbb{Z}$-braid band move.

\begin{defn}  \label{partqbbm} \rm
A {\it parted $\mathbb{Q}$-braid band move} is defined to be a geometric $\mathbb{Q}$-braid band move between parted mixed braids, such that it takes place at the top part of the braid and on the right of the rightmost strand, $s_k$, of the specific surgery component, $s$, consisting of the strands $s_1, \ldots , s_k$. Moreover, the little band starts from the last strand of the moving subbraid and it moves over each moving strand and each component of the surgery braid, until it reaches the last strand of $s$, and then is followed by parting of the resulting mixed braid. See Figures~\ref{partedqbbm} and \ref{partecables}.
\end{defn}

\begin{figure}
\begin{center}
\includegraphics[width=6in]{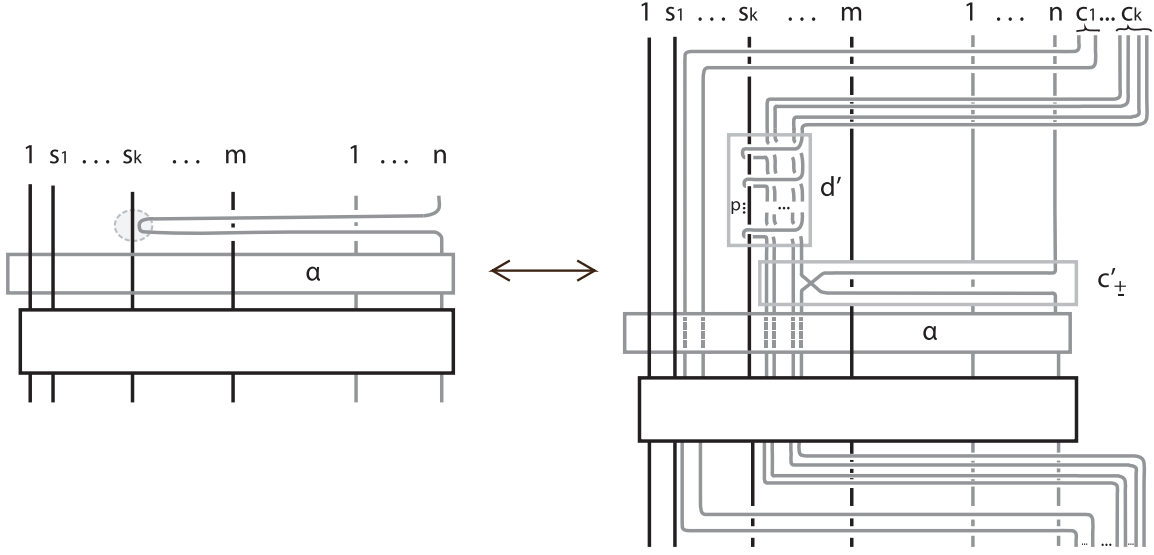}
\end{center}
\caption{ A parted $\mathbb{Q}$-braid band move.}
\label{partedqbbm}
\end{figure}

\begin{figure}
\begin{center}
\includegraphics[width=5.9in]{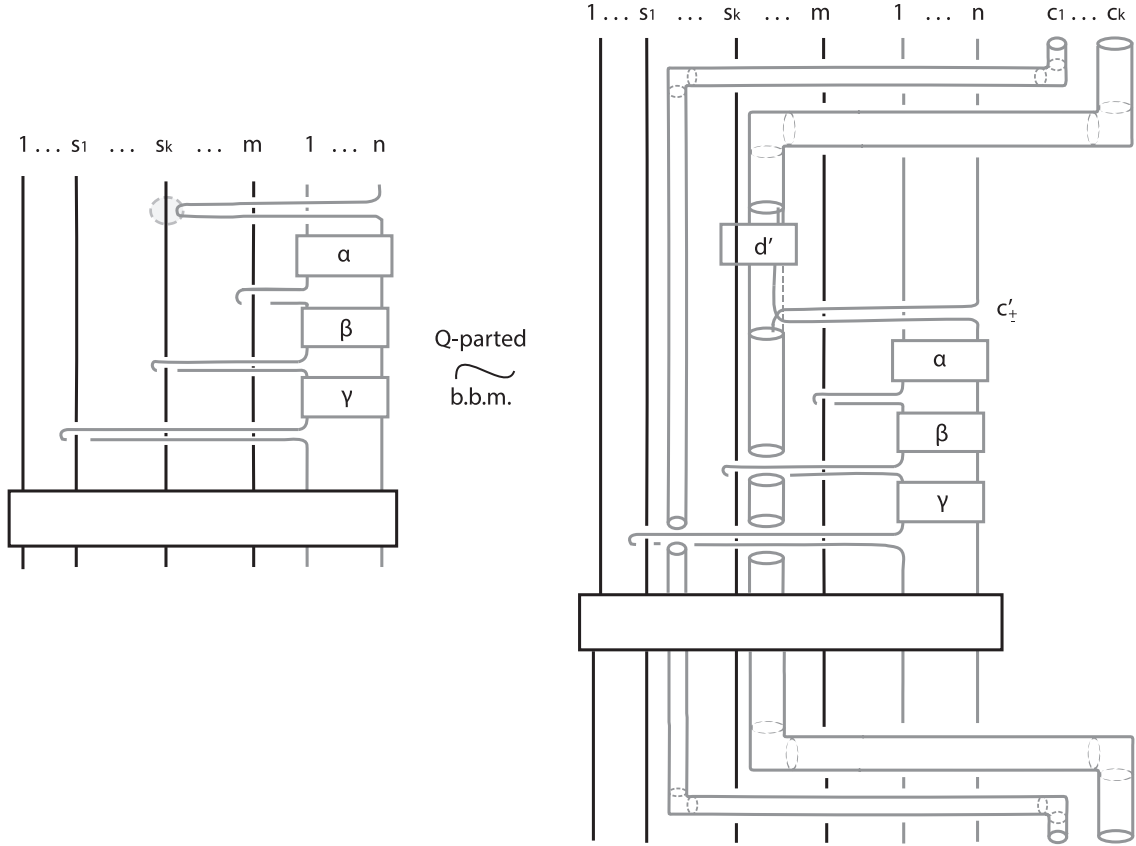}
\end{center}
\caption{ A parted $\mathbb{Q}$-braid band move using cables.}
\label{partecables}
\end{figure}

Then Theorem~\ref{geommarkovrat} restricts to the following.

\begin{thm}[Parted version of braid equivalence for $M=\chi_{_{\mathbb{Q}}} (S^3, \widehat B)$] \label{qparted} \ Two
oriented links in $M=\chi_{_{\mathbb{Q}}} (S^3, \widehat B)$ are isotopic if and only if  any two
corresponding parted mixed braids in $C_{m,\infty}$ differ by a finite sequence of parted
$L$-moves, loop  conjugations and parted $\mathbb{Q}$-braid band moves.
\end{thm}

\begin{proof}
By Lemma~\ref{partcable} the cables resulting from a geometric $\mathbb{Q}$-braid band move are treated as one strand, so we can apply Theorem~\ref{partedcco}. Moreover, by Lemma~9 in \cite{LR2} a geometric $\mathbb{Q}$-braid band move may be always assumed, up to $L$-equivalence, to take place on the right of the rightmost strand of the specific surgery component.
\end{proof}

In order to translate Theorem~\ref{qparted} into an algebraic equivalence between elements of $B_{m,\infty}$ we need the following lemmas.

\begin{lemma}[Combing Lemma, Lemma 6 \cite{LR2}] \label{combinglem}
{ \ The crossings $\Sigma_k$, $k=1,\ldots,m-1$ of the fixed subbraid $B$, and the loops $a_i$,  for
$i=1,\ldots,m$, satisfy the following {\rm `combing' relations}:
\[
\begin{array}{llcll}
 \ \  & \Sigma_k {a^{\pm 1}_k}  & = & {a^{\pm 1}_{k+1}} \Sigma_k  &    \\

 \ \  & \Sigma_k {a^{\pm 1}_{k+1}}  & = &  {a^{-1}_{k+1}} {a^{\pm 1}_k} a_{k+1}
\Sigma_k  &   \\

 \ \  & \Sigma_k {a^{\pm 1}_i}   & = &  {a^{\pm 1}_i}  \Sigma_k  & \mbox{if  \ }
 i \neq k, k+1 \\

 \ \  & {\Sigma^{-1}_k} {a^{\pm 1}_k}  & = & a_k {a^{\pm 1}_{k+1}} {a^{-1}_k}
 {\Sigma^{-1}_k}  &    \\

 \ \  & {\Sigma^{-1}_k} {a^{\pm 1}_{k+1}}  & = & {a^{\pm 1}_k}
{\Sigma^{-1}_k}  &   \\

  & {\Sigma^{-1}_k} {a^{\pm 1}_i}   & = &  {a^{\pm 1}_i}  {\Sigma^{-1}_k}
 & \mbox{if  \ }  i \neq k, k+1.  \\
 \end{array}
\]
 }
\end{lemma}

\noindent \textbf{Notation:} We set $\lambda_{k,r}:=\sigma_k \sigma_{k+1} \ldots \sigma_{r-1} \sigma_r$, for $k<r$ and $\lambda_{k,r}:=\sigma_k \sigma_{k-1} \ldots \sigma_{r+1} \sigma_r$, for $r<k$.
We note that $\lambda_{i,i}:=\sigma_i$. Also, by convention we set $\lambda_{0,i}=\lambda_{i,0}:=1$.

Then we have the following:

\begin{lemma} \label{cableloop}
A positive looping between a $q$-strand cable and the $j^{th}$ fixed strand of the fixed subbraid $B$ has the algebraic expression:
$$\prod_{i=0}^{q-1} \lambda_{i,1} a_j\lambda_{i,1}^{-1}=\prod_{i=0}^{q-1} \lambda_{1,(q-1)-i}^{-1} a_j\lambda_{1,(q-1)-i} \ ,$$
\noindent while a negative looping has the algebraic expression:
$$\prod_{i=0}^{q-1} \lambda_{1,i}^{-1} a_j^{-1} \lambda_{1,i}=\prod_{i=0}^{q-1} \lambda_{(q-1)-i,1} a_j^{-1} \lambda_{(q-1)-i,1}^{-1} \ .$$
\end{lemma}

\begin{proof}
We start with Figure~\ref{lemma2}(a) where a positive looping between a $q$-strand cable and a fixed stand of the mixed braid is shown. In Figure~\ref{lemma2}(b) the cable is replaced by the $q$ strands according to Definition~\ref{cable}. Then, using mixed braid isotopy, we end up with Figure~\ref{lemma2}(c), top, whereby we can read directly the algebraic expression $\prod_{i=0}^{q-1} \lambda_{i,1} a_j \lambda_{i,1}^{-1}$. The second algebraic expression comes from the bottom illustration of Figure~\ref{lemma2}. Similarly, in Figure~\ref{lemma2b} we illustrate the case where a negative looping between a $q$-strand cable and a fixed strand of the mixed braid occurs.
\end{proof}

\begin{figure}
\begin{center}
\includegraphics[width=5.8in]{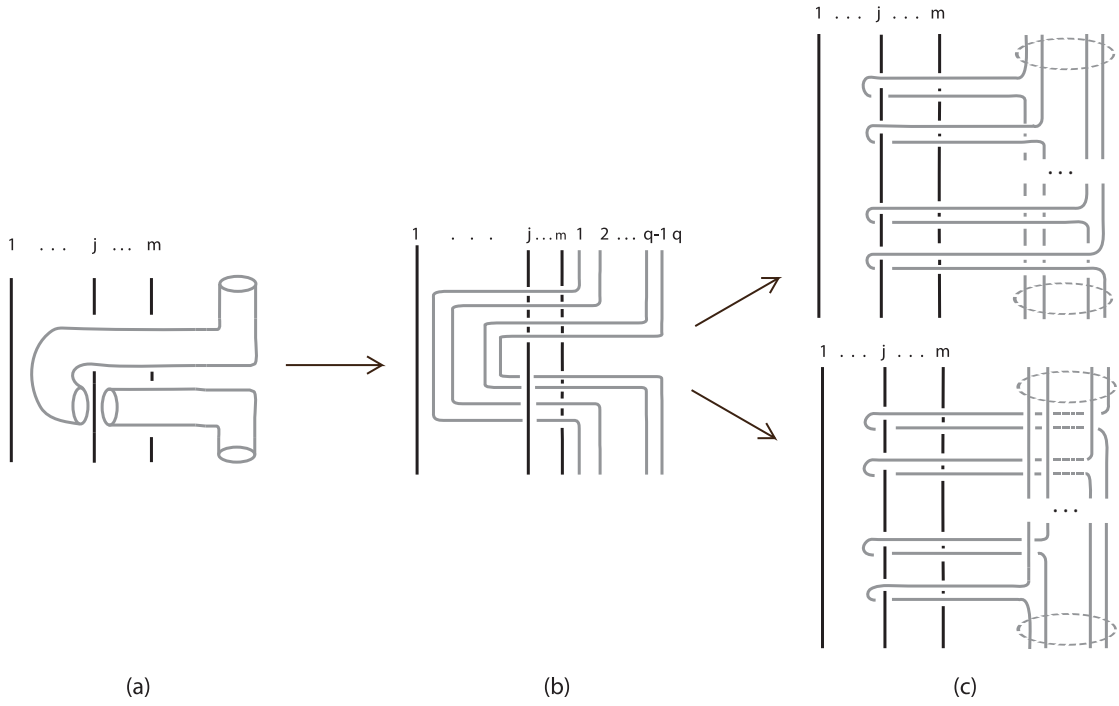}
\end{center}
\caption{ A positive looping between a cable and a fixed strand.}
\label{lemma2}
\end{figure}

\begin{figure}
\begin{center}
\includegraphics[width=5.8in]{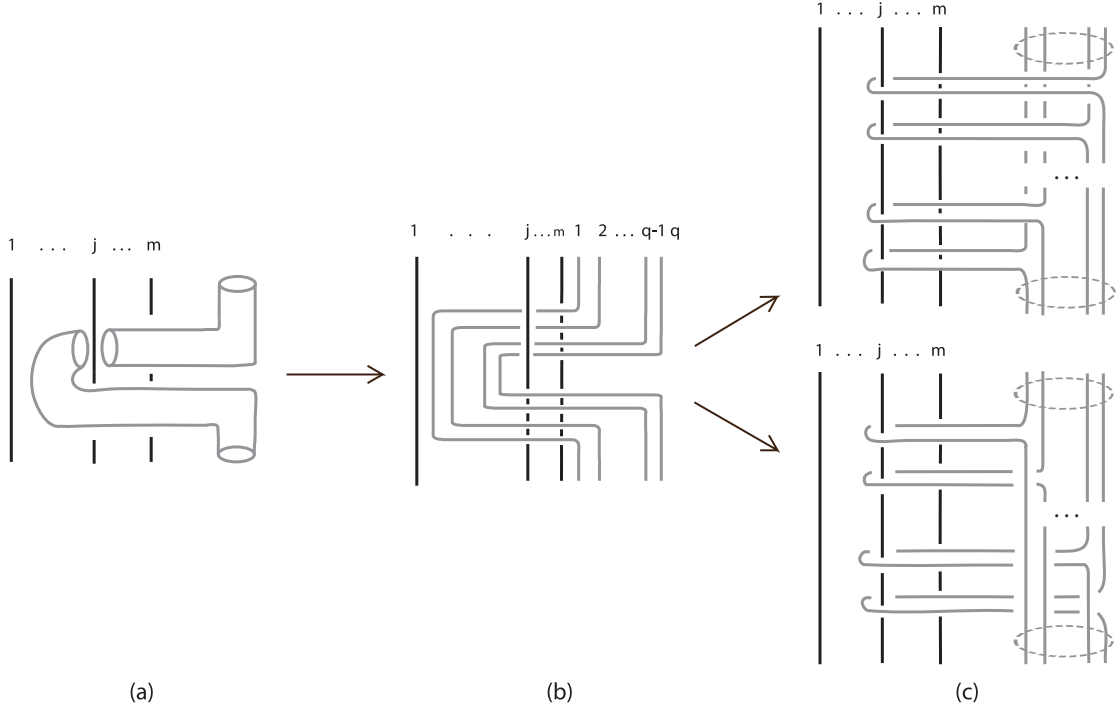}
\end{center}
\caption{ A negative looping between a cable and a fixed strand.}
\label{lemma2b}
\end{figure}

\begin{lemma} \label{cablecomb}
Cabling and combing commute. That is, treating a $q$-strand cable as a thickened moving strand and combing it through the fixed subbraid $B$, the result is equivalent to combing one by one each strand of the cable.
\end{lemma}

\begin{proof}
According to the Combing Lemma we have to consider all cases between looping and crossings of the subbraid $B$. We will only examine the four cases illustrated in Figure~\ref{combing} as representative cases. All others are completely analogous.
The first case is illustrated in Figure~\ref{combcable1}, where a positive looping between the cable and the $k^{th}$ fixed strand of $B$ is being considered and the crossing of the fixed strands is positive. For a negative looping the proof is similar.

\begin{figure}
\begin{center}
\includegraphics[width=5.8in]{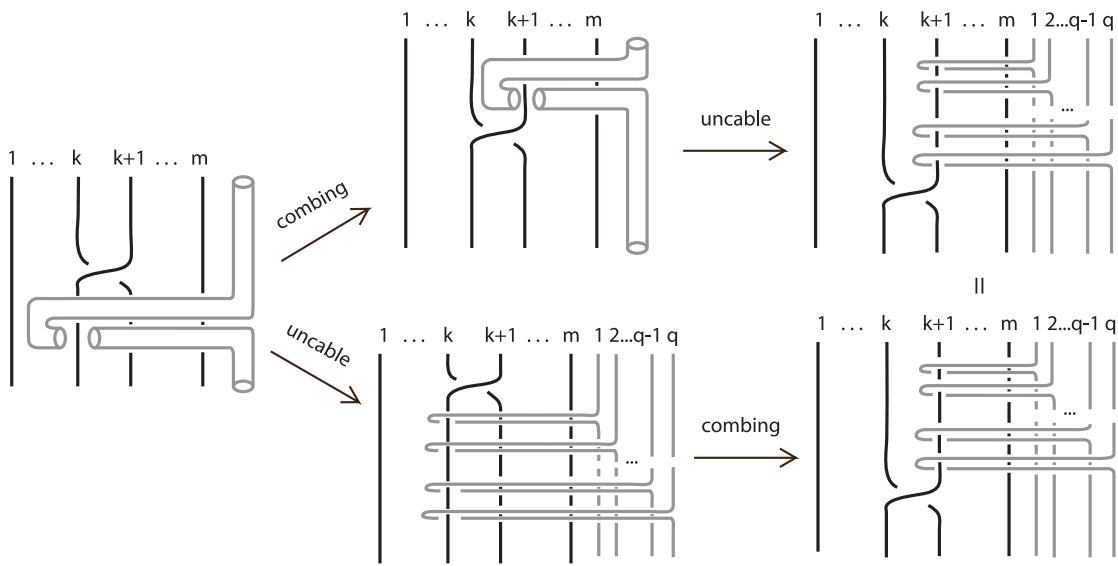}
\end{center}
\caption{ Combing and cabling commute: Proof of Case 1.}
\label{combcable1}
\end{figure}

We now consider the case illustrated in Figure~\ref{combcable2}, where a positive looping between the cable and the $(k+1)^{th}$ fixed strand of $B$ is being considered, and the crossing in $B$ is positive. We shall prove this case by induction on the number of strands that belong to the cable, since, as we can see from Figure~\ref{combcable2}, the resulting algebraic expressions are not directly comparable.

\begin{figure}
\begin{center}
\includegraphics[width=5.8in]{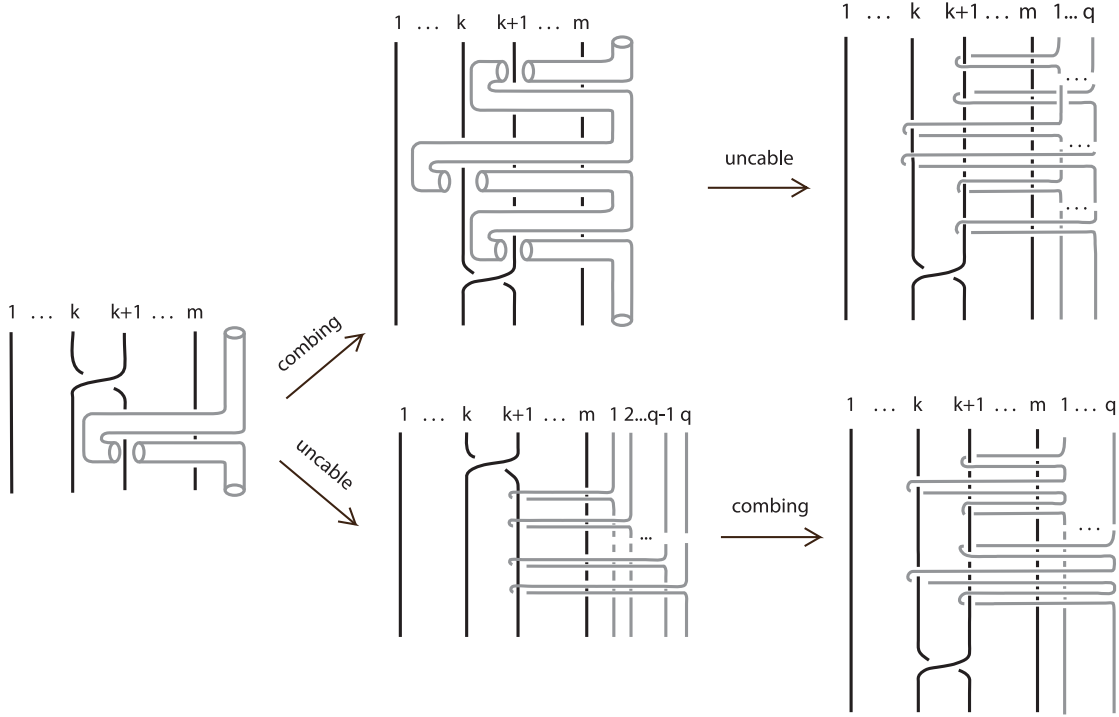}
\end{center}
\caption{ Combing and cabling commute: Case 2. }
\label{combcable2}
\end{figure}

The case where the cable consists of one strand is trivial. For a two-strand cable, combing the cable first and then uncabling (see top part part of Figure~\ref{2sccomb}) results in the algebraic expression:
$$
\alpha_2^{-1}\ (\sigma_1^{-1} \alpha_2^{-1} \sigma_1)\ \alpha_1\ (\sigma_1^{-1} \alpha_1 \sigma_1)\ \alpha_2\ (\sigma_1^{-1} \alpha_2 \sigma_1),
$$
while uncabling first and then combing (bottom part of Figure~\ref{2sccomb}) results in the algebraic expression:
$$(\alpha_2^{-1}\alpha_1\alpha_2)\ (\sigma_1^{-1} \alpha_2^{-1}\alpha_1\alpha_2 \sigma_1).$$

\begin{figure}
\begin{center}
\includegraphics[width=5.8in]{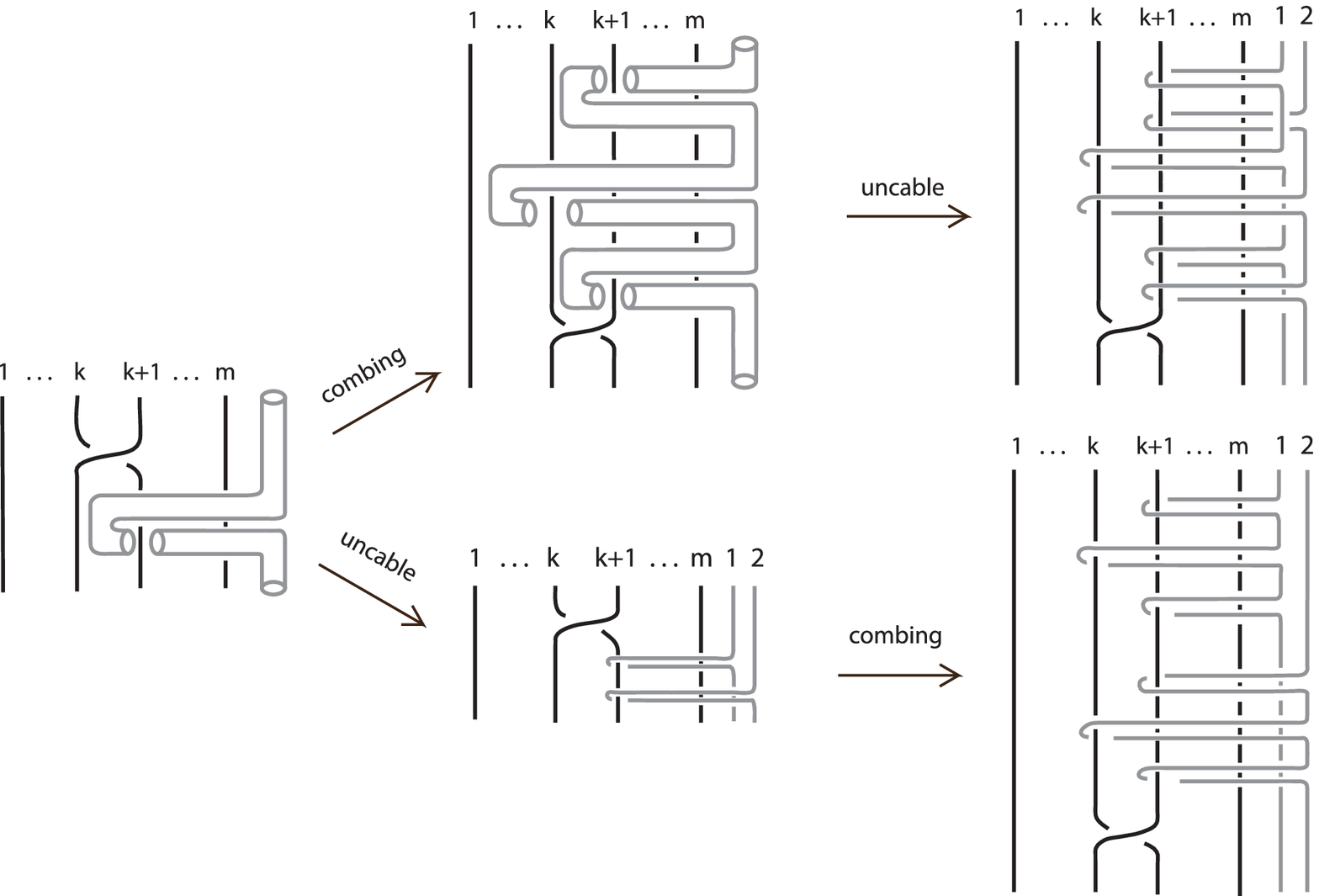}
\end{center}
\caption{ Combing a $2$-strand cable: Case 2. }
\label{2sccomb}
\end{figure}

We show below that these algebraic expressions are equal, whereby we have underlined expressions which are crucial for the next step. Indeed:


$$
\begin{array}{rclr}
\underline{\alpha_2^{-1}} (\sigma_1^{-1} \alpha_2^{-1} \sigma_1) \alpha_1 (\sigma_1 \alpha_1 \sigma_1^{-1}) \alpha_2 (\sigma_1 \underline{\alpha_2 \sigma_1^{-1}}) & = & (\underline{\alpha_2^{-1}}\alpha_1\alpha_2)(\sigma_1 \alpha_2^{-1}\alpha_1\underline{\alpha_2 \sigma_1^{-1}}) & \Leftrightarrow \\
(\sigma_1^{-1} \alpha_2^{-1} \sigma_1) \alpha_1 (\underline{\sigma_1 \alpha_1 \sigma_1^{-1}) \alpha_2} (\sigma_1)& = & (\alpha_1\alpha_2)(\sigma_1 \alpha_2^{-1}\alpha_1) & \Leftrightarrow \\
\sigma_1^{-1} \alpha_2^{-1} \sigma_1 \alpha_1  \alpha_2 \sigma_1 \underline{\alpha_1\sigma_1^{-1}\sigma_1}& = & \alpha_1\alpha_2\sigma_1 \alpha_2^{-1}\underline{\alpha_1}&\Leftrightarrow \\
\sigma_1^{-1} \underline{\alpha_2^{-1} \sigma_1 \alpha_1 (\sigma_1^{-1}} \sigma_1) \alpha_2 \sigma_1 & = & \alpha_1\alpha_2\sigma_1 \alpha_2^{-1}&\Leftrightarrow \\
\underline{\sigma_1^{-1} \sigma_1 \alpha_1} \sigma_1^{-1} \alpha_2^{-1} \sigma_1 \alpha_2 \sigma_1  & = & \underline{\alpha_1} \alpha_2\sigma_1 \alpha_2^{-1}&\Leftrightarrow \\
\underline{\sigma_1^{-1} \alpha_2^{-1}} \sigma_1 \alpha_2 \sigma_1  & = &  \alpha_2\sigma_1 \underline{\alpha_2^{-1}}&\Leftrightarrow \\
\sigma_1 \alpha_2  \sigma_1 \alpha_2  & = &  \alpha_2\sigma_1 \alpha_2 \sigma_1 \\
\end{array}
$$
\noindent We ended up with one of the defining relations of the  mixed braid group $B_{m,n}$, recall (\ref{B}).

We now consider a $(q+1)$-strand cable and we let the first $q$ strands form a $q$-strand subcable. We first comb the $q$-strand cable and then the $(q+1)^{st}$ strand and the result follows by applying the case of a 2-strand cable and the induction hypothesis for the $q$-strand cable (see Figure~\ref{combcableproof2}).
\end{proof}

\begin{figure}
\begin{center}
\includegraphics[width=6in]{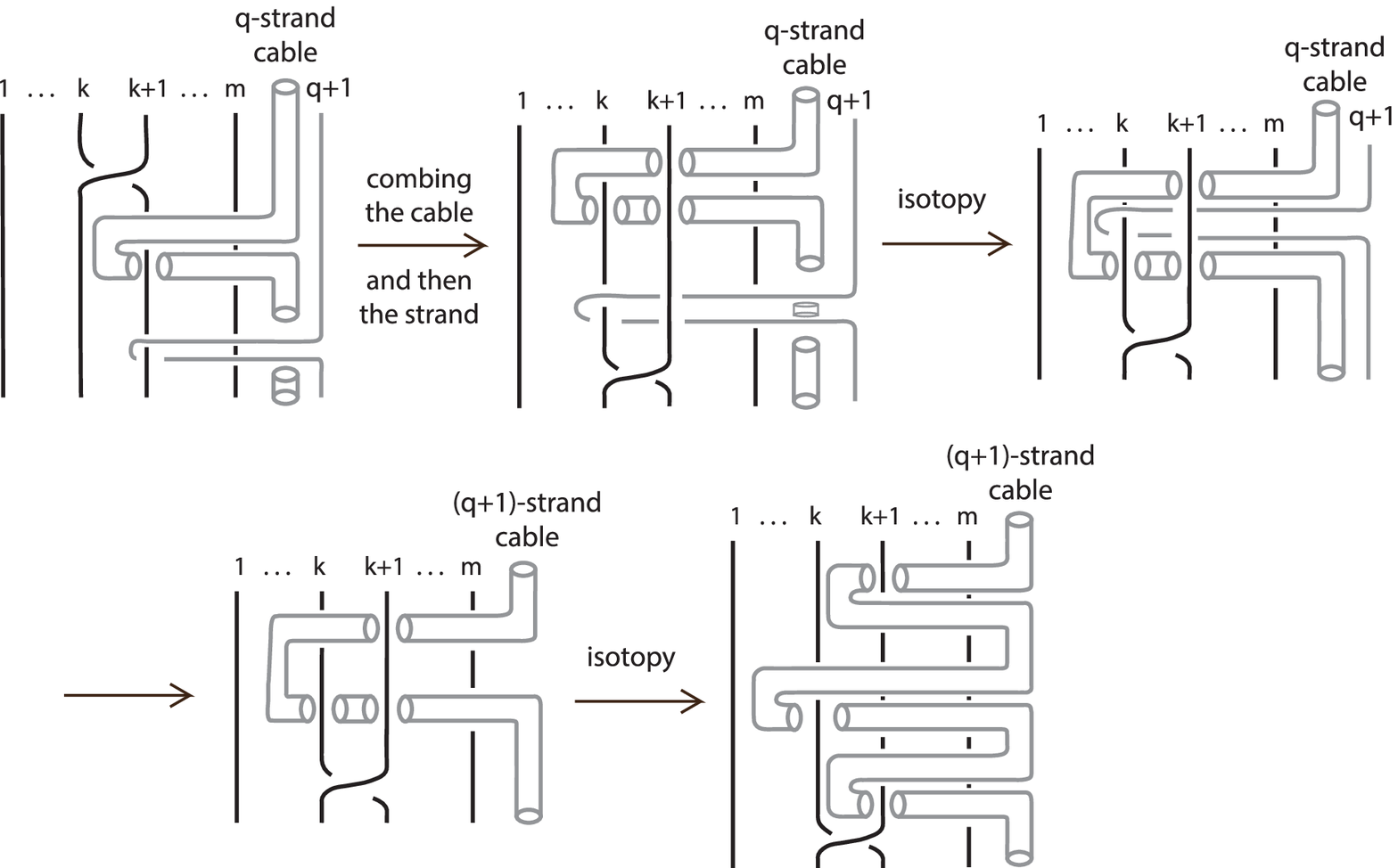}
\end{center}
\caption{ Combing and cabling commute: Proof of Case 2.}
\label{combcableproof2}
\end{figure}

Let now $B\bigcup \beta$ be a parted mixed braid and let a parted $\mathbb{Q}$-braid band move be performed on the last strand, $s_k$, of a surgery component consisting of the strands $s_1, \ldots , s_k$. Recall Figures~\ref{partedqbbm} and \ref{partecables}. In order to give an algebraic expression for the parted $\mathbb{Q}$-braid band move, we part locally the subbraids $d^{\prime}$ and $c_{\pm}^{\prime}$ and the loop generators $a _i$, $i=1, \ldots, m$, and we use mixed braid isotopy in order to transform $d^{\prime}$ into $d$ and $c_{\pm}^{\prime}$ into $c_{\pm}$. See Figures~\ref{cables2}, \ref{cablesc}, \ref{cablesa} and \ref{cablesa2}.
Then, $d$ has the algebraic expression:
\begin{equation}
d\ =\ [\ \lambda_{n+kq-1,n+(k-1)q+1}\ \lambda_{n+1,n+(k-1)q}^{-1}\ \lambda_{n,1}\ a_{s_k}\ \lambda_{n,1}^{-1}\ \lambda_{n+1,n+(k-1)q}^{-1}\ ]^p
\end{equation}
and $c_{\pm}$ has the algebraic expression:
\begin{equation}
c_{\pm}\ =\ \lambda_{n,n+kq-2}\ \sigma_{n+kq-1}^{\pm1}\ \lambda_{n,n+kq-2}^{-1}.
\end{equation}

\begin{figure}
 \begin{center}
\includegraphics[width=4.7in]{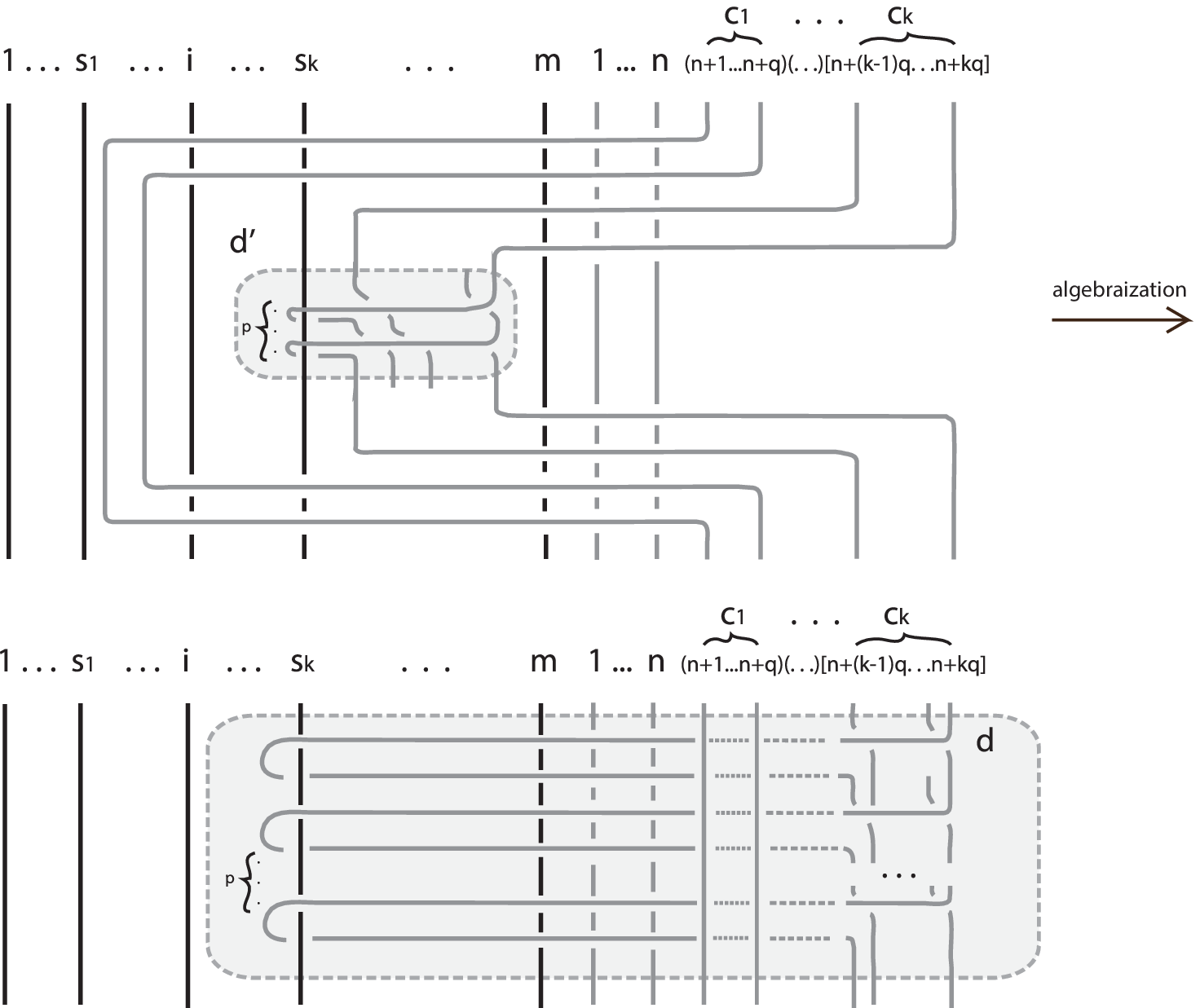}
\end{center}
\caption{The parting of a $\mathbb{Q}$-braid band move is an algebraic $\mathbb{Q}$-braid band move. }
\label{cables2}
\end{figure}

\begin{figure}
\begin{center}
\includegraphics[width=3.7in]{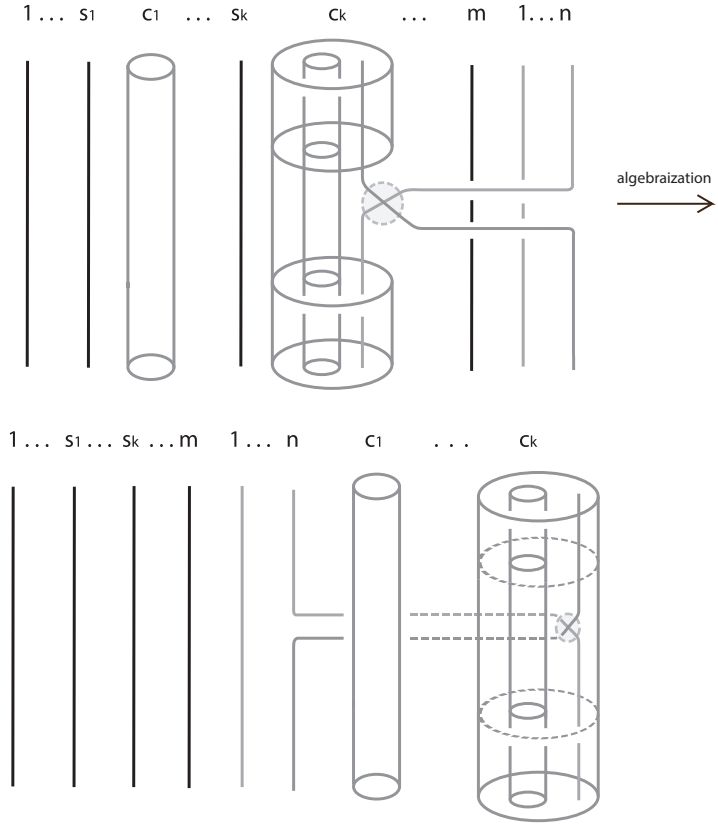}
\end{center}
\caption{Algebraization of the crossing part $c^{\prime}_{\pm}$ of the mixed braid to $c_{\pm}$ after a $\mathbb{Q}$-braid band move is performed. }
\label{cablesc}
\end{figure}

\begin{figure}
\begin{center}
\includegraphics[width=5.8in]{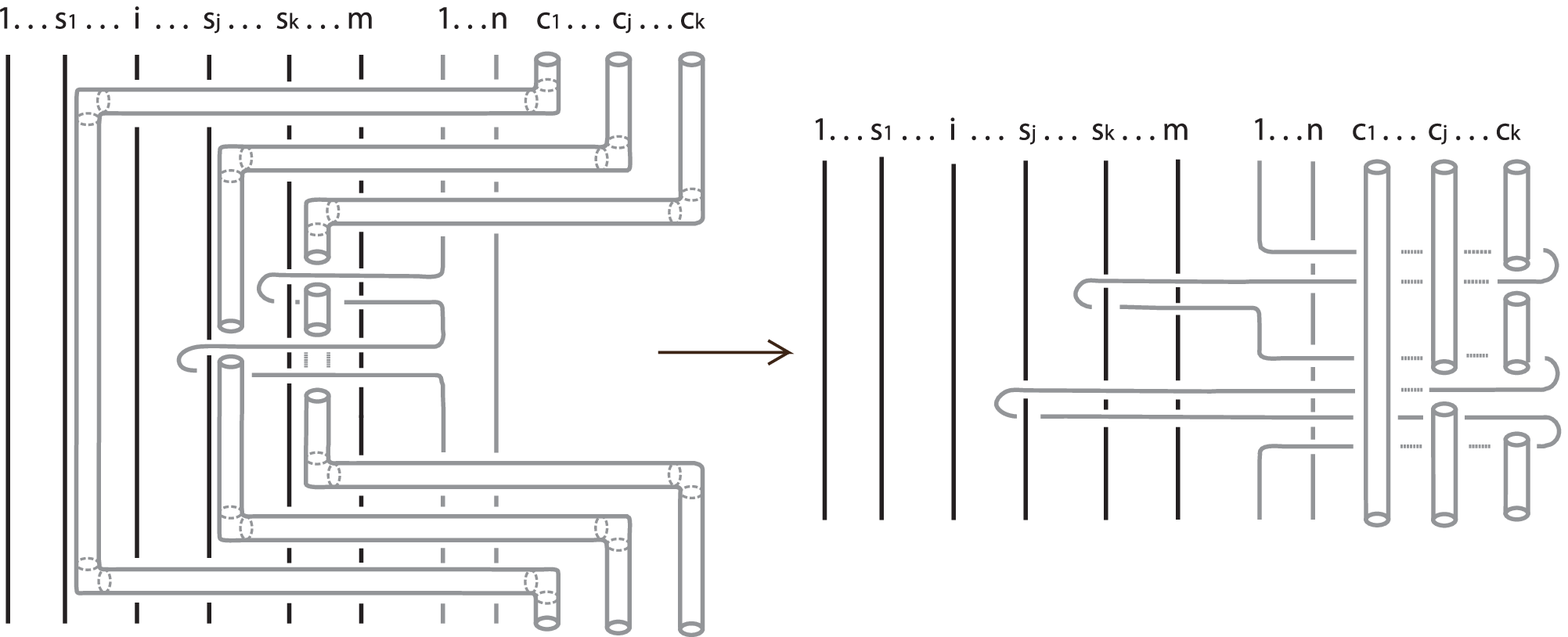}
\end{center}
\caption{Algebraization of the loop generators $a_j$ for $j \in \{s_1, \ldots, s_k \}$ after a $\mathbb{Q}$-braid band move is performed. }
\label{cablesa}
\end{figure}

\begin{figure}
\begin{center}
\includegraphics[width=6in]{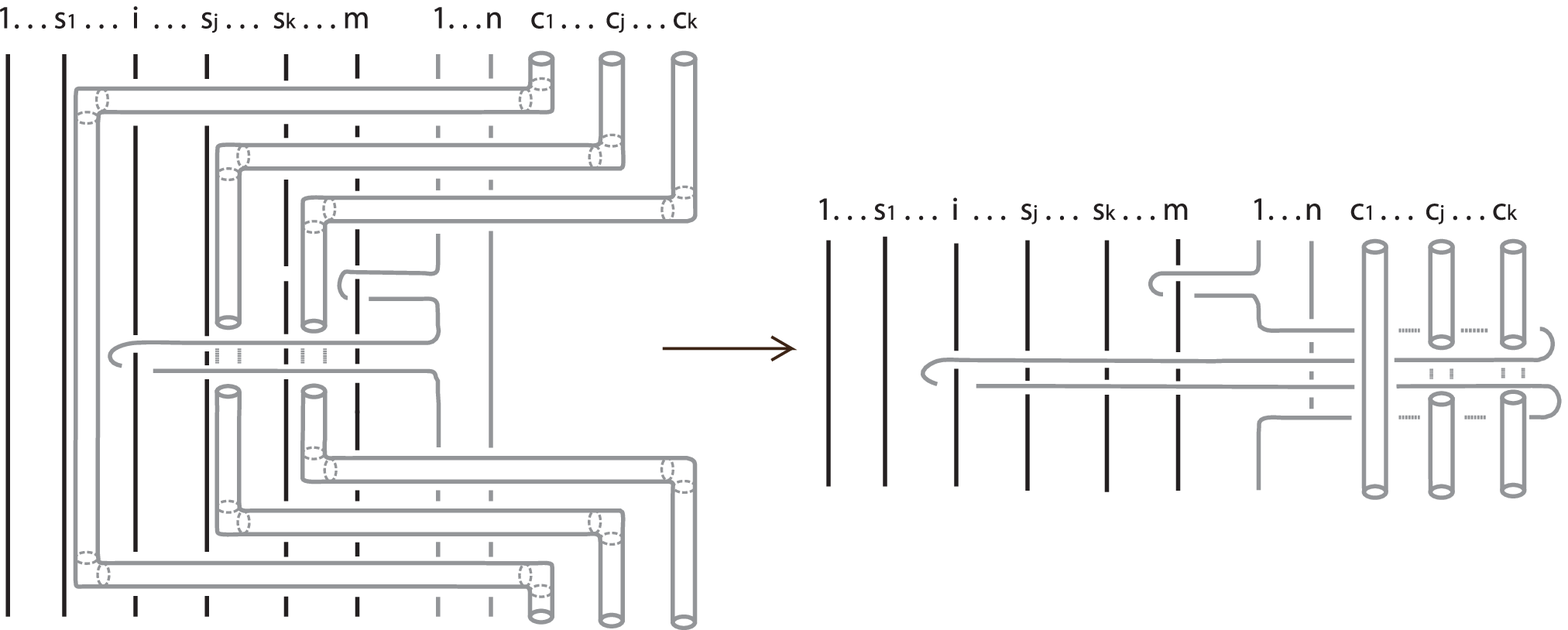}
\end{center}
\caption{Algebraization of the loop generators $a_j$ for $j \notin \{s_1, \ldots, s_k \}$ after a $\mathbb{Q}$-braid band move is performed. }
\label{cablesa2}
\end{figure}

We are now in the position to give the definition of an algebraic $\mathbb{Q}$-braid band move.

\begin{defn} \label{algbm}
\rm
(i) An {\it algebraic $\mathbb{Q}$-braid band move} is defined to be a parted $\mathbb{Q}$-braid band move between elements of $B_{n,\infty}$ and it has the following algebraic expression:
$$
\beta \sim d \ c_{\pm} \ \beta^{\prime},
$$

\noindent where $\beta^{\prime}$ is the algebraic mixed braid $\beta$ with the substitutions:
$$
\begin{array}{ccl}
{a_i}^{\pm 1} & \longleftrightarrow & {a_i}^{\pm 1}, \mbox{ \ for} \ i > s_k, \\
&& \\
{a_i}^{\pm 1} & \longleftrightarrow &  \lambda^{-1}_{n-1,1}  \lambda_{n,n+kq-1}  \lambda_{n+kq-1,1} \ {a_i}^{\pm 1} \\
 & & \lambda^{-1}_{n-1,1} \lambda^{-1}_{n+kq-1,n}  \lambda^{-1}_{n,n+kq-1}  \lambda_{n-1,1}, \mbox{\ for} \ i < s_1,\\
\end{array}
$$

$$
\begin{array}{ccl}
a_{s_{j}}  & \longleftrightarrow &  \lambda^{-1}_{n-1,1}  \lambda_{n,n+kq-1}  \lambda_{n+kq-1,n+(j-1)q}  \lambda^{-1}_{n,n+(j-1)q-1}  \lambda_{n-1,1} \ a_{s_{j}}\\
& &  \lambda^{-1}_{n-1,1} \lambda_{n,n+jq-1} \lambda^{-1}_{n+kq-1,n+jq} \lambda^{-1}_{n,n+kq-1} \lambda_{n-1,1} \\
&& \\
&  & and \\
&& \\
a^{-1}_{s_{j}} & \longleftrightarrow & \lambda^{-1}_{n-1,1} \lambda_{n,n+kq-1} \lambda_{n+kq-1,n+jq} \lambda^{-1}_{n+(j-1)q,n+jq-1} \lambda^{-1}_{n,n+(j-1)q-1} \lambda_{n-1,1} \ a^{-1}_{s_{j}}\\
& & \lambda^{-1}_{n-1,1} \lambda_{n,n+(j-1)q-1} \lambda^{-1}_{n+jq-1,n+(j-1)q} \lambda^{-1}_{n+kq-1,n+jq}
\lambda^{-1}_{n+kq-1,n} \lambda_{n-1,1}, \\
&&  \mbox{ \ for} \ s_j \in \{s_1, \ldots, s_k \},\\
&& \\
a^{\pm 1}_{j} & \longleftrightarrow & \lambda^{-1}_{n-1,1} \lambda_{n,n+kq-1} \lambda_{n+kq-1,n+(r-1)q} \lambda^{-1}_{n,n+(r-1)q-1} \lambda_{n-1,1} a^{\pm 1}_{j} \\
&& \lambda^{-1}_{n-1,1} \lambda_{n,n+(r-1)q-1} \lambda^{-1}_{n+kq-1,n+(r-1)q} \lambda^{-1}_{n,n+kq-1} \lambda_{n-1,1}, \mbox{ \ for} \ s_{r-1}<j<s_r.\\
\end{array}
$$

(ii) A {\it combed algebraic $\mathbb{Q}$-braid band move\/} is a move between algebraic mixed braids and is defined to be a parted $\mathbb{Q}$-braid band move that has been combed through $B$. Moreover, it has the following algebraic expression:
$$
\beta \sim d \  c_{\pm} \ \beta^{\prime} \ comb_B(c_1, \ldots , c_k),
$$

\noindent where $comb_B(c_1, \ldots , c_k)$ is the combing of the parted $q$-strand cables $c_1, \ldots , c_k$ through the surgery braid $B$ (see Figure~\ref{cables}).
\end{defn}

\begin{figure}
\begin{center}
\includegraphics[width=6.1in]{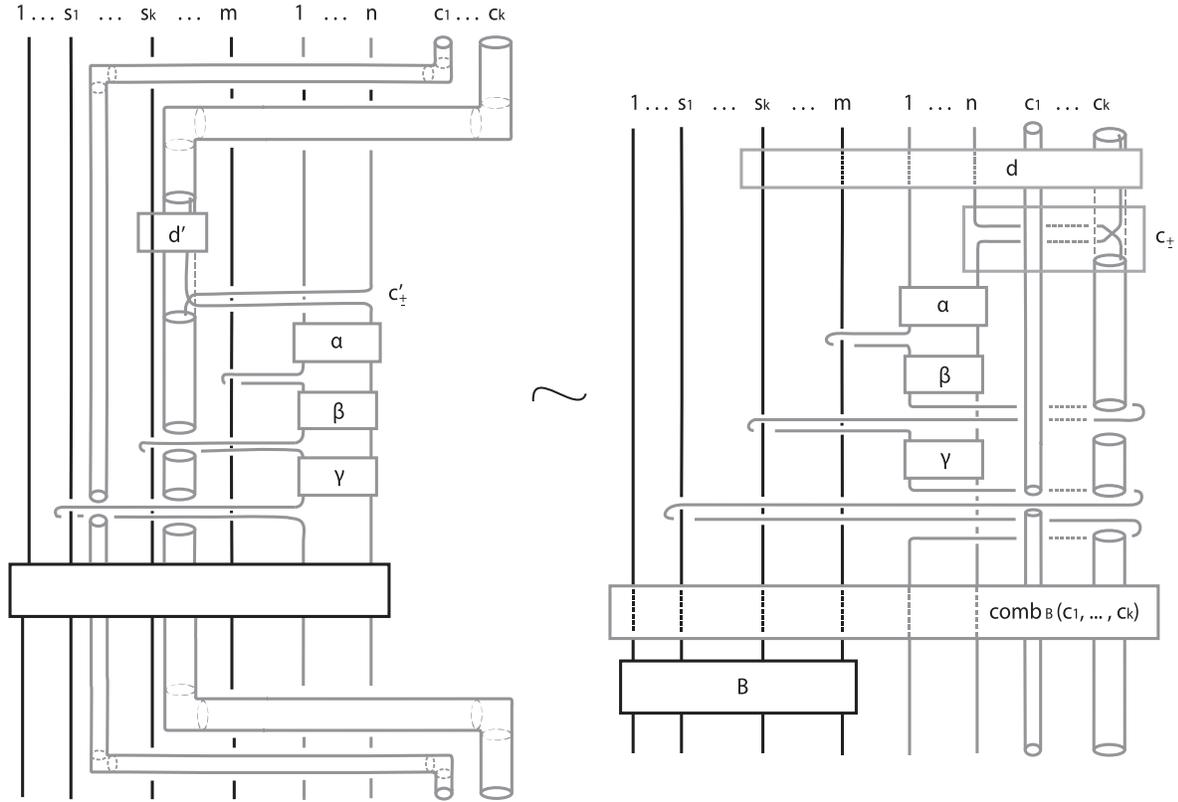}
\end{center}
\caption{ A combed algebraic $\mathbb{Q}$-braid band move. }
\label{cables}
\end{figure}

We are, finally, in the position to state the following main result of the paper.

\begin{thm}[Algebraic Markov Theorem for  $M=\chi_{_{\mathbb{Q}}}(S^3, \widehat{B})$]  \label{algmarkov}
Let $s_1, \ldots , s_k$ be the numbers of the strands of a surgery component $s$ and let $c_1, \ldots , c_k$ be the corresponding $q$-strand cables arising from a $\mathbb{Q}$-braid band move performed on $s$. Then, two oriented links in $M=\chi_{_{\mathbb{Q}}}(S^3, \widehat{B})$ are isotopic if and only if any two corresponding algebraic mixed braid representatives in $B_{m,\infty}$ differ by a finite sequence of the following moves: \\
(i) \ Algebraic $M$-moves: \ $\beta_1 \beta_2\sim \beta_1 \sigma_n^{\pm1} \beta_2$, for $\beta_1, \beta_2 \in B_{m,n}$, \\
(ii) \ Algebraic $M$-conjugation: \ $\beta \sim \sigma_j^{\mp1} \beta \sigma_j^{\pm1}$, for $\beta, \sigma_j \in B_{m,n}$, \\
(iii) \ Combed loop conjugation: $\beta \sim \alpha_i^{\mp1} \beta \ \rho_i^{\pm1}$, for $\beta \in B_{m,n}$, where $\rho_i$
is the combing of the loop $\alpha_i$ through $B$, \\
(iv) \ Combed algebraic braid band moves: $\beta \ \sim \ d \ c_{\pm} \ \beta^{\prime} \ comb_{B}(c_1, \ldots, c_k)$, where the algebraic expressions of $d$ and $c_{\pm}$ are as in Eqs.~(3) and (4) respectively, $\beta^{\prime}$ is $\beta$ with the substitutions of the loop generators as in Definition~\ref{algbm} and $comb_{B}(c_1, \ldots, c_k)$ is the combing of the resulting $q$-strand cables $c_1, \ldots, c_k$ through the fixed subbraid $B$,

\smallbreak

\noindent or equivalently, by a finite sequence of the following moves: \\
$\bullet$ algebraic $L$-moves (see algebraic expressions in Eqs~2), \\
$\bullet$ combed loop conjugation, \\
$\bullet$ combed algebraic braid band moves (Definition~\ref{algbm}).
\end{thm}

\begin{proof}
The arguments for passing from parted braid equivalence (Theorem~\ref{qparted}) to algebraic braid equivalence are the same as in those in the proof of the transition from Theorem~\ref{partedcco} to Theorem~\ref{algcco} in the case of integral surgery. The only part we need to analyze in detail is the algebraization of a parted $\mathbb{Q}$-braid band move. Namely, we will show that the following diagram commutes.
\begin{center}
$$
\begin{CD}
C_{m,n} \ni B \bigcup \beta  @>{\text{Parted Q-b.b.m.}}>> B\bigcup \beta^{\prime} \in C_{m,n+kq} \\
@|  @| \\
comb_B{\beta} @. comb_B{\beta^{\prime}} \\
@VVV  @VVV \\
B_{m,n} \ni alg_B(\beta) @>{\text{\text{Algebraic Q-b.b.m.}}}>> alg_B(\beta^{\prime}) \in B_{m+n+kq}
\end{CD}
$$
\end{center}
In words, we start with a parted mixed braid $B\bigcup \beta \in C_{m,n}$ and we perform on it a parted $\mathbb{Q}$-braid band move (Definition~\ref{partqbbm}) obtaining a parted mixed braid $B\bigcup \beta^{\prime} \in C_{m,n+kq}$, where $k$ is the number of strands forming the surgery component. We then comb both parted mixed braids obtaining $comb_B(\beta)$ and $comb_B(\beta^{\prime})$ respectively. We will show that the corresponding algebraic parts, $alg_B(\beta) \in B_{m,n}$ and $alg_B(\beta^{\prime}) \in B_{m,n+kq}$ differ by the algebraic braid equivalence given in the statement of the theorem.
We apply Lemma~8 in \cite{LR2}, where the $q$ strands of a braid band move are placed in the cable and the cable is treated as one strand. More precisely, we note that the parted $\mathbb{Q}$-braid band move takes place at the top of the braid, so it forms an algebraic $\mathbb{Q}$-braid band move. We now comb away $\beta$ to the top of $B$ and on the other side we comb away $\beta^{\prime}$. Since the $q$-strands cable of the parted $\mathbb{Q}$-braid band move lie very close to the surgery strands, this ensures that the loops $\alpha_j^{\pm1}$ around any strand of the $k$ strands of the specific surgery components get  combed in the same way before and after the $\mathbb{Q}$-braid band move. So, having combed away $\beta$ we are left at the bottom with the identity moving braid on the one hand, and with the combing of all cables of the braid band move on the other hand, which is precisely what we denote $comb_B()$. Finally, by Lemma~\ref{cablecomb}, combing and cable commute. Thus, the Theorem is proved.
\end{proof}

\section{Examples}
In this section we give the braid equivalences for knots in specific $3$-manifolds that play a very important role to $3$-dimensional topology
such as the lens spaces $L(p,q)$, homology spheres and Seifert manifolds.

\subsection{Lens spaces $L(p,q)$}

It is known that the lens spaces $L(p,q)$ can be obtained by surgery on the unknot with surgery coefficient $p/q$. So, the fixed braid
$\widehat{B}$ that represents $L(p,q)$ is the identity braid of one single strand and thus, no combing is needed. We have the following:
\smallbreak
\textit{Two oriented links in $L(p,q)$ are isotopic if and only if any two
corresponding algebraic mixed braids in $B_{1, \infty}$ differ by a finite sequence of the following
moves (compare with \cite{LR2}, Section~4) :}

\smallbreak

{\it (1)} \ {\it Algebraic $M$-moves: \ $\beta \sim \beta {\sigma^{\pm 1}_n},  \ \ \ \
\beta \in B_{1,n}$}

\vspace{.03in}

{\it (2)} \ {\it Algebraic $M$-conjugation: \
$\beta \sim {\sigma^{\mp 1}_i} \beta {\sigma^{\pm 1}_i}, \ \ \ \ \beta, \sigma_i \in B_{1,n}$}

\vspace{.03in}

{\it (3)} \ {\it Loop conjugation: \
$\beta \sim {t}^{\mp 1}\ \beta \ {t}^{\pm 1}, \ \ \ \ \beta \in B_{1,n}$ }

\vspace{.03in}

{\it (4)} \ {\it Algebraic braid band moves: } \ For $\beta \in B_{1,n}$ we have:
\[
\beta \ \sim \ {d}  \, {c_{\pm}} \, \beta^{\prime},
\]
\noindent  where:
$$
\begin{array}{ccl}
d &= & [\lambda_{n+q-1,1}\ t\ \lambda^{-1}_{1,n+q-1}]^p,\\
&&\\
c_{\pm} & = & \lambda_{n,n+q-1}\ \sigma_{n+q-1}^{\pm1}\ \lambda^{-1}_{n,n+q-1},\\
\end{array}
$$

\noindent and where $\beta^{\prime} \in B_{1,n+q}$ is the word $\beta$ with the substitutions:
$$
\begin{array}{ccl}
t & \longleftrightarrow & ({\lambda^{-1}_{n-1,1}}\ {\lambda_{n,n+q-1}}\ {\lambda_{n+q-1,1}})\ t,\\
&&\\
{t}^{-1} & \longleftrightarrow &  t^{-1}\ ({\lambda^{-1}_{n+q-1,1}}\  {\lambda^{-1}_{n,n+q-1}}\ {\lambda_{n-1,1}}).\\
\end{array}
$$

\begin{figure}
\begin{center}
\includegraphics[width=6in]{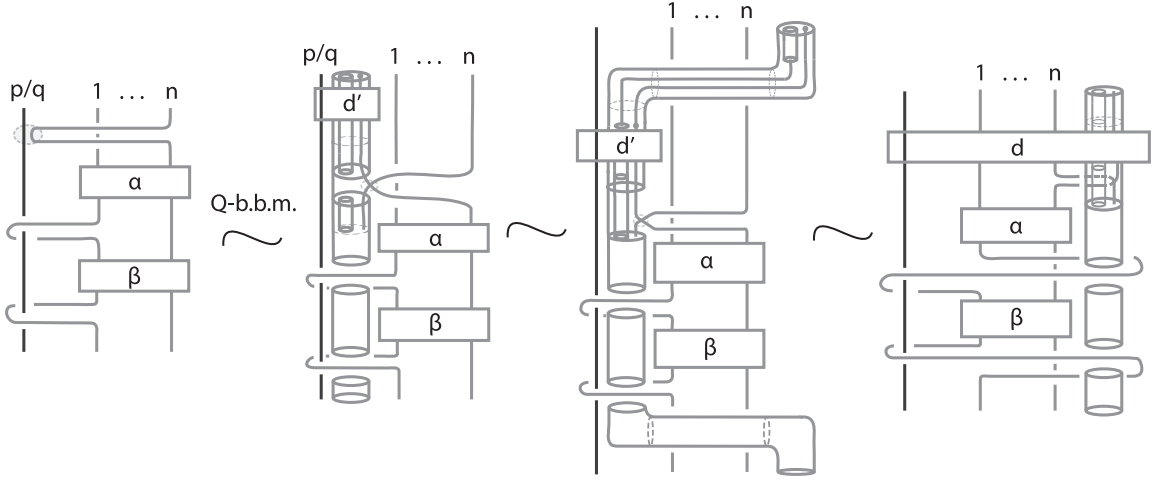}
\end{center}
\caption{ A $\mathbb{Q}$-braid band move in $L(p,q)$ and its algebraic expression. }
\label{figure20}
\end{figure}

In Figure~\ref{lpqeg} the case where $p=2$ and $q=3$ is illustrated.

\begin{figure}
\begin{center}
\includegraphics[width=3.3in]{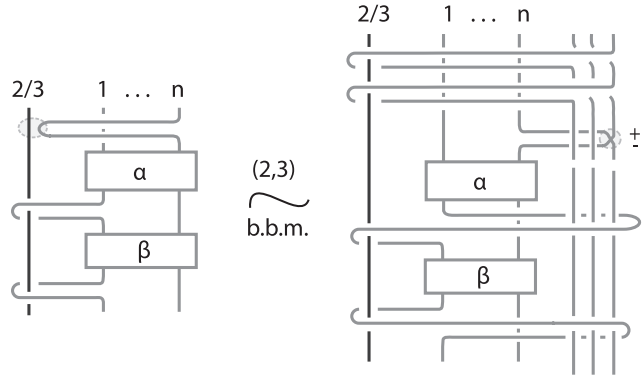}
\end{center}
\caption{ An algebraic $\mathbb{Q}$-braid band move in $L(2,3)$. }
\label{lpqeg}
\end{figure}

\subsection{Homology spheres}

It is known that a Dehn surgery on a knot yields a homology sphere exactly when the surgery coefficient is the reciprocal of an integer (see [Ro] p.262).
For example, surgery on the right-handed trefoil, with surgery coefficient $-1$ yields the Poincare Manifold also known as ``dodecahedral space" (for the algebraic braid equivalence in this case see \cite{LR2}, Section~4). In this subsection we give the algebraic braid equivalence for knots in a homology sphere $M$ obtained from $S^3$ by surgery on the trefoil knot with rational surgery coefficient $1/q$, where $q \in \mathbb{Z}$.

\smallbreak

\textit{Two oriented links in $M$ are isotopic if and only if any two corresponding algebraic mixed braids in $B_{2, \infty}$ differ by a finite sequence of the following moves:}

\smallbreak

{\it (1)} \ {\it Algebraic $M$-moves: \ $\beta \sim \beta {\sigma^{\pm 1}_n},  \ \ \ \
\beta \in B_{2,n}$}

\vspace{.03in}

{\it (2)} \ {\it Algebraic $M$-conjugation: \
$\beta \sim {\sigma^{\mp 1}_i} \beta {\sigma^{\pm 1}_i}, \ \ \ \ \beta, \sigma_i \in B_{2,n}$}

\vspace{.03in}

{\it (3)} \ {\it Combed Loop conjugation: \
$\beta \sim {a_i}^{\mp 1} \beta {\rho_i}^{\pm 1}, \ \ \ \ \beta \in B_{2,n}$, \textit{where} $\rho_i$ \textit{is the combing of the loop} $a_i$ \textit{through} $\widehat{B}$,}

\vspace{.03in}

{\it (4)} \ {\it Combed algebraic braid band moves: } \ $\beta \sim d \ c_{\pm} \ \beta^{\prime} \ comb_B(c_1, c_2)$, \textit{where:} $\beta \in B_{2,n}$,

$$
\begin{array}{ccl}
d & = & (\lambda_{n+2q-1,n+q+1}\ \lambda_{n+1,n+q}^{-1}\ \lambda_{n,1})\ a_2\ (\lambda_{n,1}^{-1}\ \lambda_{n+1,n+q}),\\
&&\\
c_{\pm} & = & \lambda_{n,n+2q-1}\ \sigma_{n+2q-1}^{\pm1}\ \lambda_{n,n+2q-1}^{-1},\\
\end{array}
$$
\noindent $\beta^{\prime}$ \textit{is the word} $\beta$ \textit{with the substitutions:}
$$
\begin{array}{ccl}
a_1 & \longleftrightarrow & (\lambda_{n-1,1}^{-1} \ \lambda_{n,n+2q-1}\  \lambda_{n+2q-1,n+q}\  \lambda^{-1}_{n,n+q-1}\  \lambda_{n-1,1}) a_1,\\
& & \\
a^{-1}_1 & \longleftrightarrow &  a^{-1}_1 \ (\lambda_{n-1,1}^{-1}\  \lambda_{n,n+q-1}\  \lambda^{-1}_{n+2q-1,n+q}\  \lambda^{-1}_{n,n+2q-1}\ \lambda_{n-1,1}),\\
& & \\
a_2 & \longleftrightarrow & (\lambda_{n-1,1}^{-1}\ \lambda_{n,n+2q-1}\ \lambda_{n+2q-1,1})\ a_2\\
& & (\lambda_{n-1,1}^{-1}\ \lambda_{n,n+q-1}\ \lambda^{-1}_{n+2q-1,n+q}\ \lambda^{-1}_{n,n+2q-1}\ \lambda_{n-1,1}),\\
& & \\
a^{-1}_2 & \longleftrightarrow & (\lambda_{n-1,1}^{-1}\ \lambda_{n,n+2q-1}\ \lambda_{n+2q-1,n+q}\ \lambda_{n,n+q}^{-1}\ \lambda_{n-1,1})\ a^{-1}_2 \\
& & (\lambda^{-1}_{n+2q-1,1}\ \lambda^{-1}_{n,n+2q-1}\ \lambda_{n-1,1}),\\
\end{array}
$$
\noindent \textit{and} $comb_B(c_1, c_2)$ \textit{is the combing of the $q$-strand cables ($c_1$ and $c_2$) through the fixed braid:}
$$
\begin{array}{ccl}
comb_B(c_1, c_2) & = & \prod_{i=0}^{q-1}{\lambda_{n+i,1} \ a_2\ \lambda^{-1}_{n+i,1}} \ \prod_{i=0}^{q-1}{\lambda_{n+2q-1-i,1}\ a^{-1}_2\ \lambda^{-1}_{n+2q-1-i,1}} \\
&&\\
&& \prod_{i=0}^{q-1}{\lambda_{n+q+i,1}\ a_1\ \lambda^{-1}_{n+q+i,1}} \ \lambda_{n+q,1} \ a_2\ \lambda^{-1}_{n,1}\ \lambda_{n+1,n+q}\\
&&\\
&& \prod_{i=1}^{q-1}{\lambda_{n+q+i,1}\ a_2\ \lambda^{-1}_{n,1}\ \lambda_{n+1,n+q}\ \lambda^{-1}_{n+q+i,n+q+1}}\\
&&\\
&& \prod_{i=0}^{q-1}{\lambda_{n+q-1-i,1}\ a_2\ \lambda^{-1}_{n+q-1-i,1}} \ \prod_{i=0}^{q-1}{\lambda_{n+i,1}\ a_1\ \lambda^{-1}_{n+i,1}}\\
&&\\
&& \prod_{i=0}^{q-1}{\lambda_{n+i,1}\ a_2\ \lambda^{-1}_{n+i,1}} \ \prod_{i=0}^{q-1}{\lambda_{n+i,1}\ a_1\ \lambda^{-1}_{n+i,1}}\\
&&\\
&& \prod_{i=0}^{q-1}{\lambda_{n+i,1}\ a_2\ \lambda^{-1}_{n+i,1}} \ \prod_{i=0}^{q-1}{\lambda_{n+q+i,n+1+i}}.\\
\end{array}
$$

\subsection{Seifert Manifolds}

It is known that a Seifert manifold $M((p_1,q_1), \ldots, (p_{m-1},q_{m-1}))$ has a rational surgery description as shown in Figure~\ref{seifsurg} (see \cite{Sa}, p.33)).

\begin{figure}
\begin{center}
\includegraphics[width=3.2in]{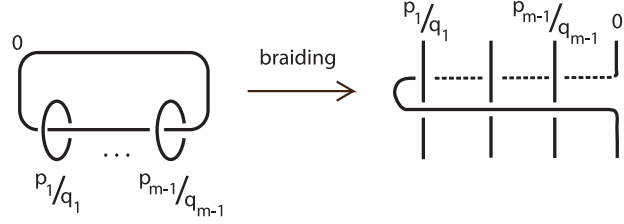}
\end{center}
\caption{ Surgery description of a Seifert manifold. }
\label{seifsurg}
\end{figure}

\textit{Two oriented links in a Seifert manifold $M((p_1,q_1), \ldots, (p_{m-1},q_{m-1}))$ are isotopic if and only if any two
corresponding algebraic mixed braids differ by a finite sequence of the following moves:}

\smallbreak

{\it (1)} \ {\it Algebraic $M$-moves: \ $\beta \sim \beta {\sigma^{\pm 1}_n},  \ \ \ \
\beta \in B_{m,n}$}

\vspace{.03in}

{\it (2)} \ {\it Algebraic $M$-conjugation: \
$\beta \sim {\sigma^{\mp 1}_i} \beta {\sigma^{\pm 1}_i}, \ \ \ \ \beta, \sigma_i \in B_{m,n}$}

\vspace{.03in}

{\it (3)} \ {\it Combed loop conjugation: \
$\beta \sim {a_j}^{\mp 1} \beta {a_j}^{\pm 1}, \ \ \ \ \beta \in B_{m,n}$. }

\vspace{.03in}

{\it (4)} \ {\it Combed algebraic braid band moves: } \ For $\beta \in B_{m,n}$ we distinguish the cases:

\smallbreak

$\bullet$ If a $\mathbb{Q}$-braid band move is performed on the $j^{th}$ strand of the fixed braid with rational coefficient $p/q$ (see Figure~\ref{seifcomb}) then: $\beta \sim d \ c_{\pm} \  \beta^{\prime} \ comb_B(c_j)$, where $comb_B(c_j)$ is the combing of the $c_j$ cable through $B$,

$$d\ =\ [\lambda_{n+q-1,1}\ \alpha_i\ \lambda_{n-1,1}^{-1}]^p \ \ {\rm and} \ \ c_{\pm}\ =\ \lambda_{n,n+q-1}\ \sigma_{n+q-1}^{-1}\ \lambda_{n,n+q-1}^{-1},$$

\smallbreak

and where $\beta^{\prime}$ is $\beta$ with the substitutions:
$$
\begin{array}{ccll}
a^{\pm 1}_i & \longleftrightarrow &  a^{\pm 1}_i,& \quad i\ > \ j,\\
&&& \\
a_i^{\pm1} & \longleftrightarrow  & \lambda_{n-1,1}^{-1} \ \lambda_{n,n+q-1}\ \lambda_{n+q-1,1}\ a_i^{\pm1}& \\
&& \lambda_{n+q-1,1}^{-1}\ \lambda_{n,n+q-1}^{-1}\ \lambda_{n-1,1},&  \quad i\ < \ j,\\
&&& \\
a_j &\longleftrightarrow & \lambda_{n-1,1}^{-1} \ \lambda_{n,n+q-1} \ \lambda_{n+q-1,1} \ a_j & \\
&&& \\
a^{-1}_j & \longleftrightarrow & a^{-1}_j \lambda_{n+q-1,1}^{-1} \ \lambda^{-1}_{n,n+q-1} \ \lambda_{n-1,1}.& \\
\end{array}
$$
\noindent $\bullet$ If a $\mathbb{Q}$-braid band move is performed on the last strand of the fixed braid with surgery coefficient $0$, then:

$$\beta \sim \sigma_{n}^{\pm 1}\ \beta^{\prime},$$
\
\noindent where $\beta^{\prime}$ is $\beta$ with the substitutions:
$$
\begin{array}{ccl}
a^{\pm 1}_{j} & \longleftrightarrow & \lambda_{n-1,1}^{-1} \ \sigma^{2}_n\ \lambda_{n-1,1} a^{\pm 1}_{j}\ \lambda_{n,1}^{-1}\ \sigma_n^{-1} \ \lambda_{n-1,1},\ \textrm{for}\ j=1, \ldots, m-1,\\
&& \\
a_{m} & \longleftrightarrow & \lambda_{n-1,1}^{-1}\ \sigma^{2}_n\ \lambda_{n-1,1}\ a_{m},\\
&& \\
a^{-1}_{m} & \longleftrightarrow & a^{-1}_{m}\ \lambda_{n-1,1}^{-1}\ \sigma^{-2}_n\ \lambda_{n-1,1}.\\
\end{array}
$$

\begin{figure}
\begin{center}
\includegraphics[width=6in]{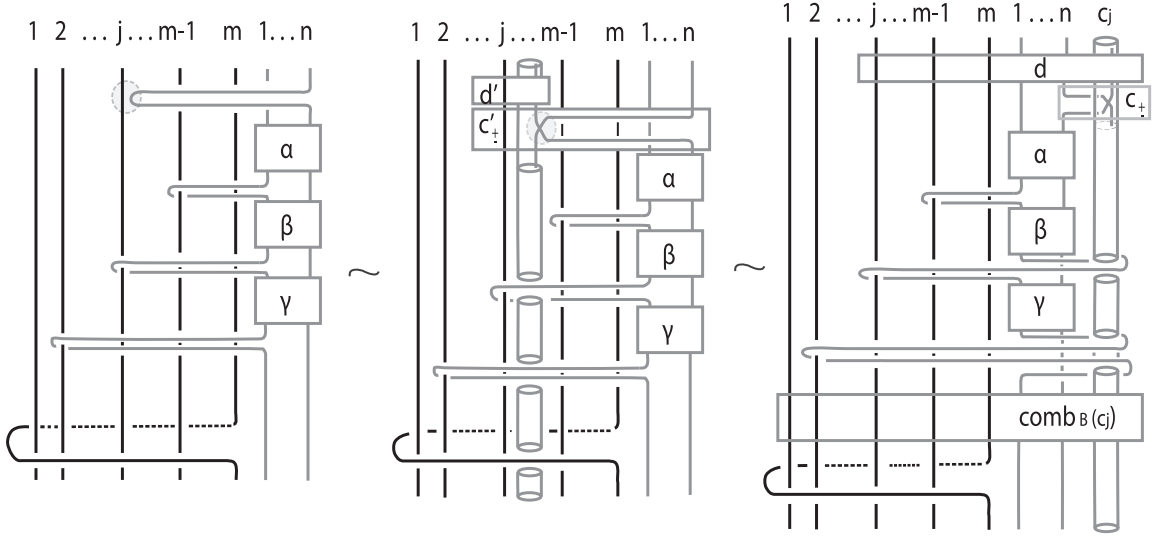}
\end{center}
\caption{ A $\mathbb{Q}$-braid band move in a Seifert manifold and its algebraic expression. }
\label{seifcomb}
\end{figure}

\subsection{Rational surgery along a torus knot}

It is well-known that a manifold $M$ obtained by rational surgery from $S^3$ along an $(m,r)$-torus knot with rational coefficient $p/q$ is either the lens space $L(|q|,pr^2)$, or the connected sum of two lens spaces $L(m,r)\sharp L(r,m)$, or a Seifert manifold (for more details the reader is referred to \cite{LM}).

\textit{Two oriented links in $M$ are isotopic if and only if any two corresponding algebraic mixed braids differ by a finite sequence of the following moves:
}
\smallbreak

{\it (1)} \ {\it Algebraic $M$-moves: \ $\beta \sim \beta {\sigma^{\pm 1}_n},  \ \ \ \
\alpha \in B_{m,n}$}

\vspace{.03in}

{\it (2)} \ {\it Algebraic $M$-conjugation: \
$\beta \sim {\sigma^{\mp 1}_i} \beta {\sigma^{\pm 1}_i}, \ \ \ \ \beta, \sigma_i \in B_{m,n}$}

\vspace{.03in}

{\it (3)} \ {\it Combed loop conjugation: \
$\beta \sim {a_i}^{\mp 1} \beta {\rho_i}^{\pm 1}, \ \ \ \ \beta \in B_{m,n}$, \textit{where} $\rho_i$ \textit{is the combing of the loop} $a_i$ \textit{through} $\widehat{B}$,}
\vspace{.03in}

{\it (4)} \ {\it Combed algebraic braid band moves:} For $\beta \in B_{m,n}$ we have:
$$\beta \sim d \ c_{\pm} \ \beta^{\prime} \ comb_B(c_1, \ldots, c_m),$$
\noindent where
$$
\begin{array}{lll}
d & = & [\ \lambda_{n+mq-1,n+(m-1)q+1}\ \lambda_{n,n+(m-1)q}^{-1}\ \lambda_{n-1,1}\ \alpha_j\ \lambda_{n-1,1}^{-1}\ \lambda_{n,n+(m-1)q}\ ]^p,\\
c_{\pm} & = & \lambda_{n,n+mq-2}\ \sigma_{n+mq-1}^{\pm1}\ \lambda_{n,n+mq-2}^{-1},
\end{array}
$$
\noindent $comb_B(c_1, \ldots, c_m)$ \textit{is the combing through the fixed braid braid of the parted moving cables parallel to the surgery strands and $\beta^{\prime}$ is the word $\beta$ with the substitutions:}
$$
\begin{array}{ccl}
a_j & \longleftrightarrow & (\lambda_{n-1,1}^{-1} \ \lambda_{n,n+mq-1} \ \lambda_{n+mq-1,n+(j-1)q} \ \lambda^{-1}_{n,n+(j-1)q-1} \ \lambda_{n-1,1}) \ a_j\\
&& (\lambda^{-1}_{n-1,1} \ \lambda_{n,n+jq-1} \ \lambda^{-1}_{n+mq-1,n+jq} \ \lambda^{-1}_{n,n+mq-1} \ \lambda_{n-1,1}),\\
&& \\
a^{-1}_j & \longleftrightarrow & (\lambda_{n-1,1}^{-1} \ \lambda_{n,n+mq-1} \ \lambda_{n+mq-1,n+jq} \ \lambda^{-1}_{n,n+jq-1} \ \lambda_{n-1,1}) \ a^{-1}_j\\
&& (\lambda^{-1}_{n-1,1} \ \lambda_{n,n+(j-1)q-1} \ \lambda^{-1}_{n+mq-1,n+(j-1)q} \ \lambda^{-1}_{n,n+mq-1} \ \lambda_{n-1,1}), \ for\ j \in \{ 1, \ldots , m \}.\\
\end{array}
$$

In Figures~\ref{homtref} and \ref{hompart} we illustrate an example where the $(m,r)$-torus knot is the $(2,3)$-torus knot, $p=2$ and $q=3$ (see Proposition~3.1 in \cite{LM} for details about the manifold obtained).

\begin{figure}
\begin{center}
\includegraphics[width=3.5in]{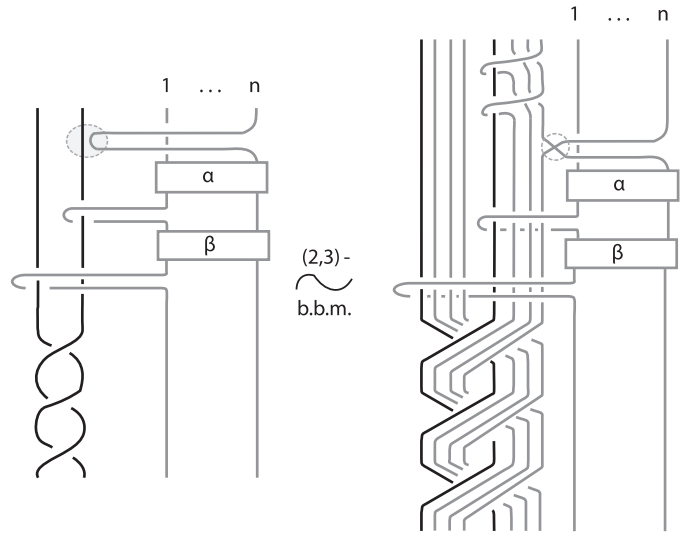}
\end{center}
\caption{ A geometric $(2,3)$-braid band move along a trefoil. }
\label{homtref}
\end{figure}

\begin{figure}
\begin{center}
\includegraphics[width=5.2in]{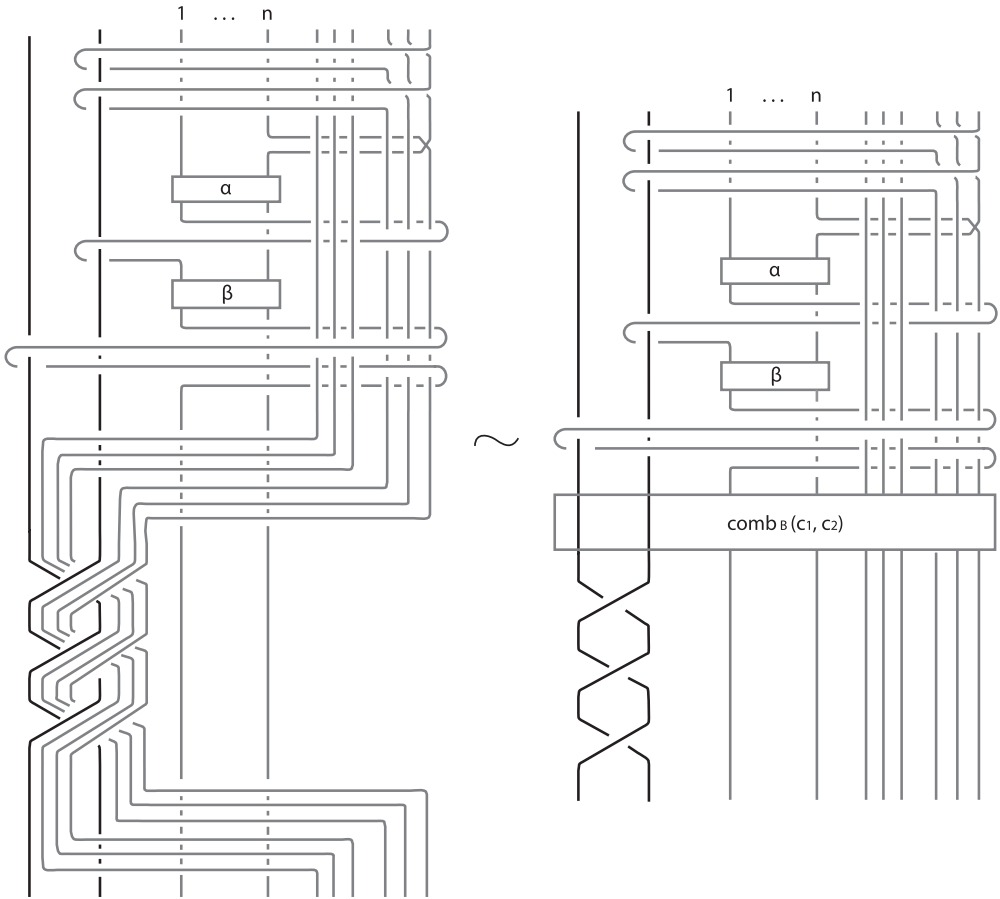}
\end{center}
\caption{ Turning the geometric $(2,3)$-braid band move into a combed algebraic $(2,3)$-braid band move. }
\label{hompart}
\end{figure}

\subsection{ Illustrations for an abstract generic example}

Let $M$ be the manifold obtained by rational surgery along a framed link $\widehat{B}$ in $S^3$. Let also $B \bigcup \beta$ be a parted mixed braid representing a link in $M$. In Figures~\ref{seifcomb1} to \ref{seifcomb6} we illustrate step-by-step the algebraization of a geometric $\mathbb{Q}$-braid band move. More precisely, in Figure~\ref{seifcomb1} a  geometric $\mathbb{Q}$-braid band move takes place on the last strand of a surgery component $(s_1,\ldots,s_k)$ of $B$.
 In Figure~\ref{seifcomb2} we part all cables $c_1,\ldots,c_k$ arising from the geometric $\mathbb{Q}$-braid band move, turning the initial  geometric $\mathbb{Q}$-braid band move  to a parted  $\mathbb{Q}$-braid band move. In Figure~\ref{seifcomb2} we also part locally the $(p,q)$-torus subbraid $d^{\prime}$.
 This leads to the algebraic expression $d$ of $d^{\prime}$, illustrated in Figure~\ref{seifcomb3}, where the local parting of the crossing subbraid $c^{\prime}_{\pm}$ is also initiated. In Figure~\ref{seifcomb4}  the algebraic expression  $c_{\pm}$ of $c^{\prime}_{\pm}$ is illustrated and the local parting of all loop generators is also initiated. This leads to the  algebraic expressions of the loop generators, in Figure~\ref{seifcomb5}, where also the preparation for combing of the cables  $c_1,\ldots,c_k$  through $B$ is illustrated. Note that the top part of Figure~\ref{seifcomb5} (above the dotted line) illustrates an algebraic $\mathbb{Q}$-braid band move. Finally, in Figure~\ref{seifcomb6} the combing of the cables $c_1,\ldots,c_k$  through $B$ is performed and the final result is a combed algebraic $\mathbb{Q}$-braid band move.
 
\begin{figure}
\begin{center}
\includegraphics[width=3.2in]{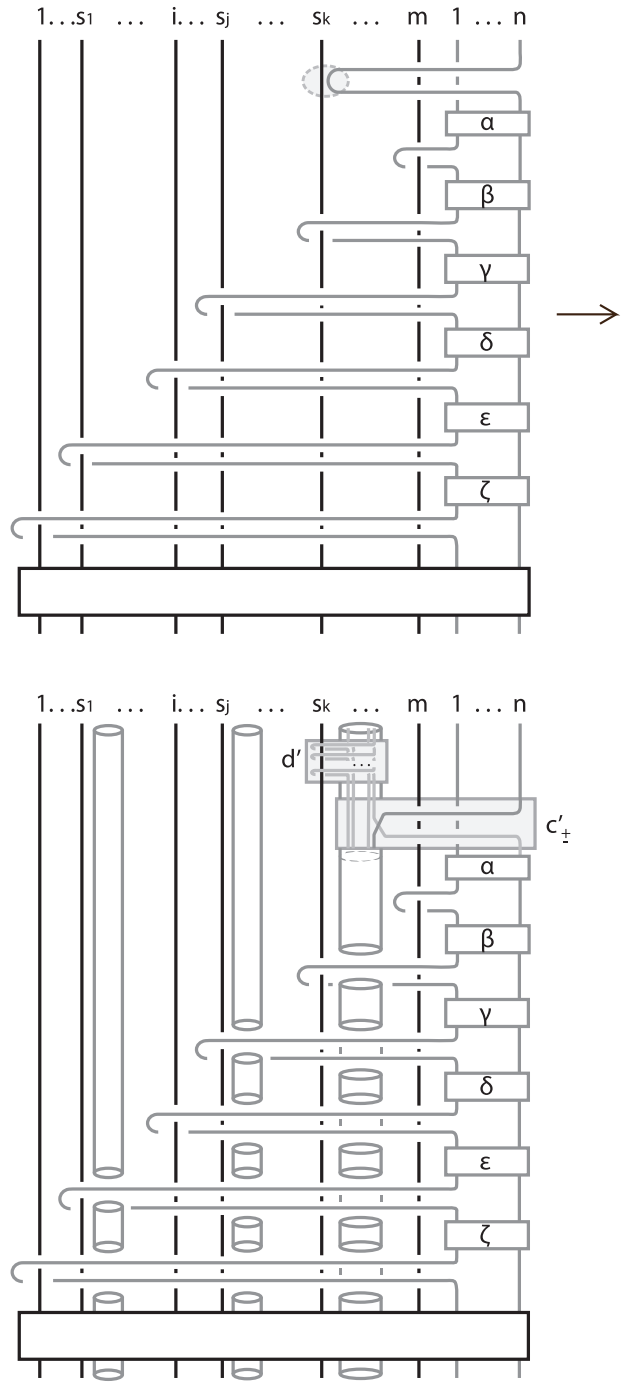}
\end{center}
\caption{ A geometric $\mathbb{Q}$-braid band move. }
\label{seifcomb1}
\end{figure}

\begin{figure}
\begin{center}
\includegraphics[width=4.1in]{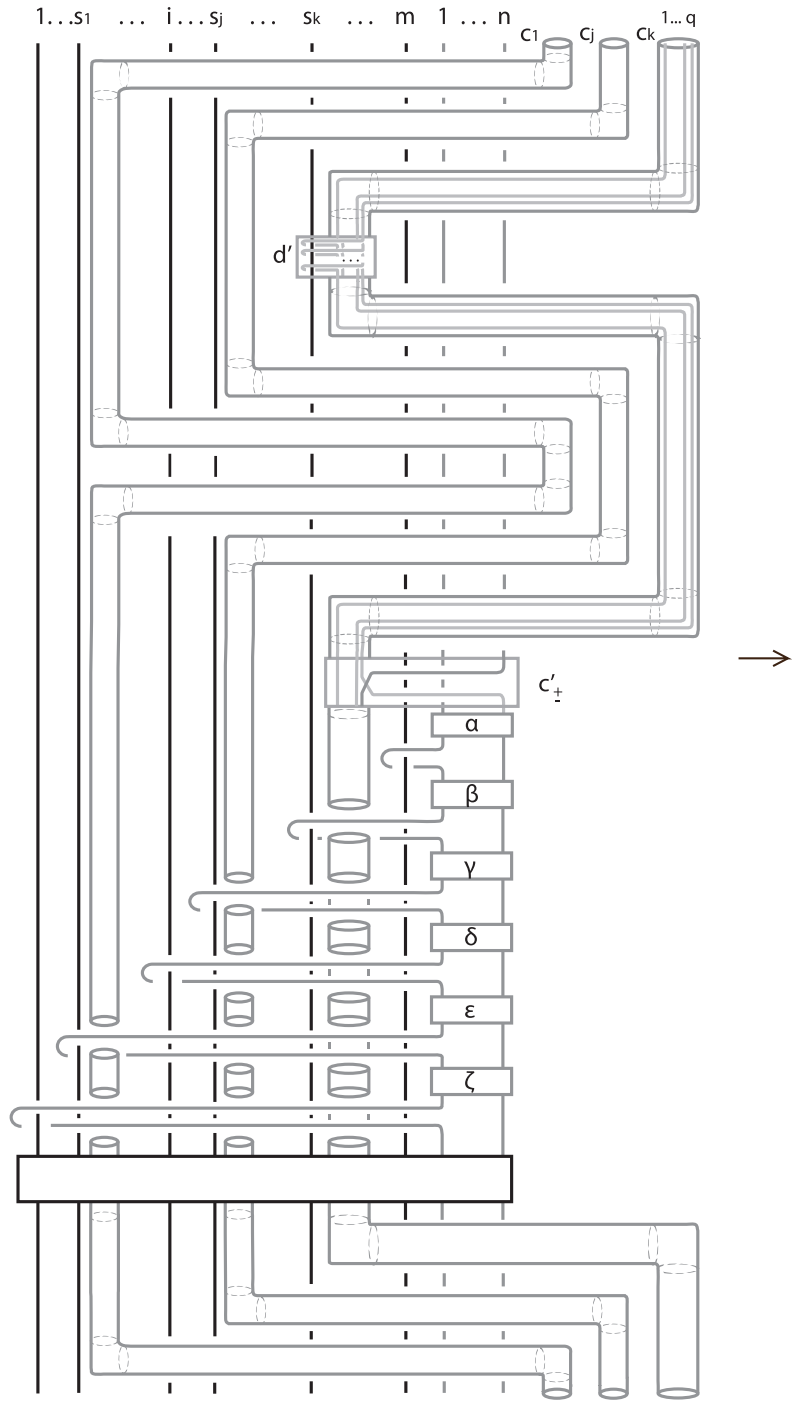}
\end{center}
\caption{ Parting locally $d^{\prime}$. }
\label{seifcomb2}
\end{figure}

\clearpage{}

\begin{figure}
\begin{center}
\includegraphics[width=3.7in]{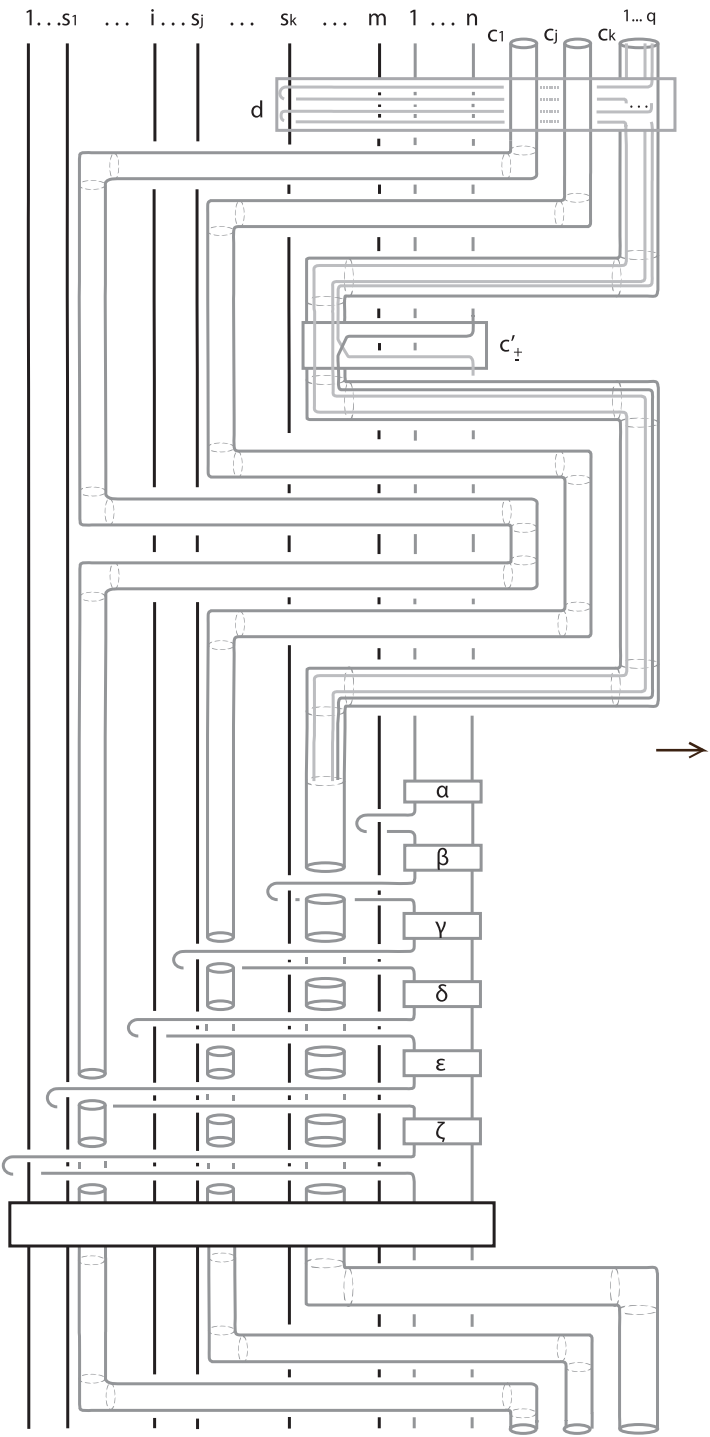}
\end{center}
\caption{ Algebraization of $d^{\prime}$ to $d$ and local parting of $c^{\prime}_{\pm}$. }
\label{seifcomb3}
\end{figure}

\begin{figure}
\begin{center}
\includegraphics[width=3.9in]{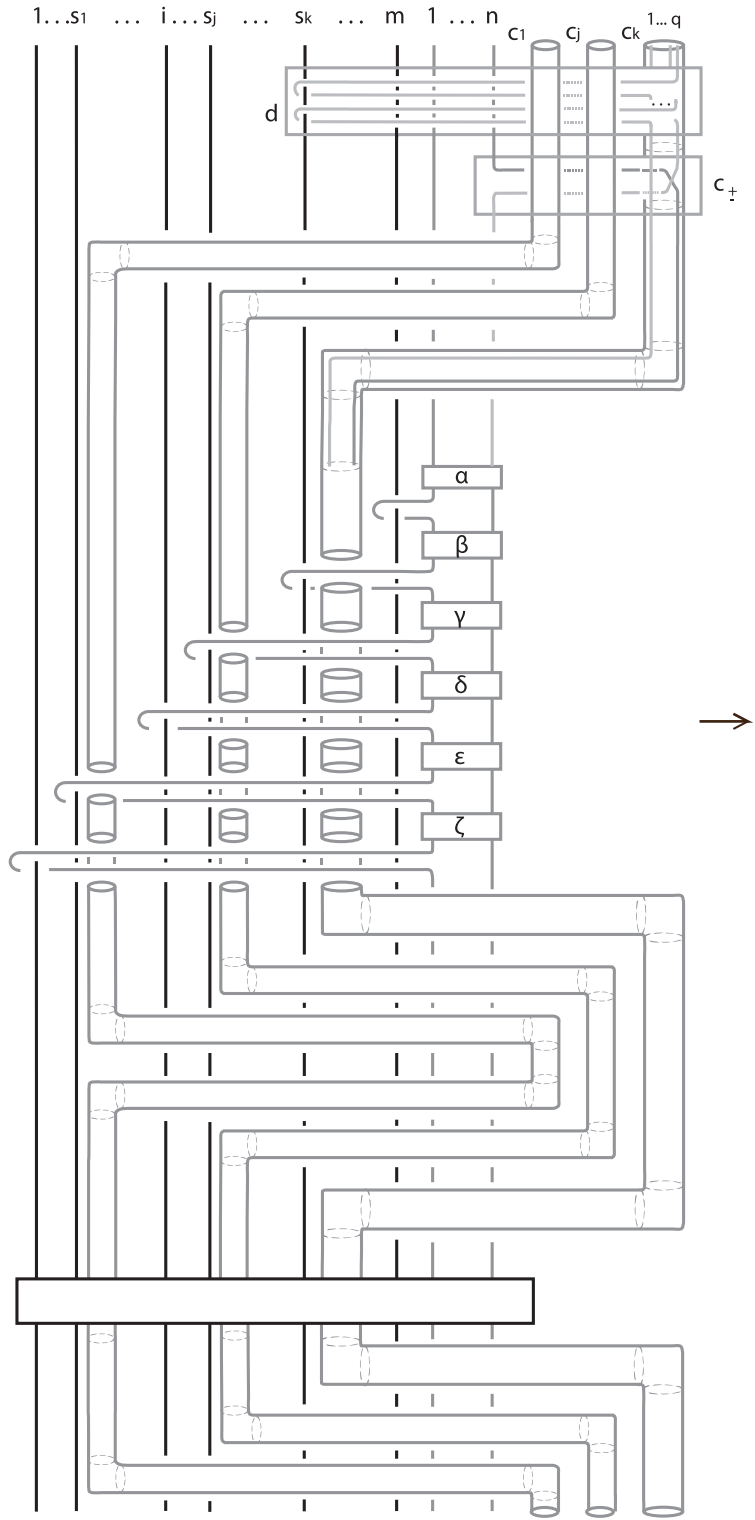}
\end{center}
\caption{ Algebraization of $c^{\prime}_{\pm}$ to $c$ and local parting of the loop generators $a_{i}$. }
\label{seifcomb4}
\end{figure}

\begin{figure}
\begin{center}
\includegraphics[width=4in]{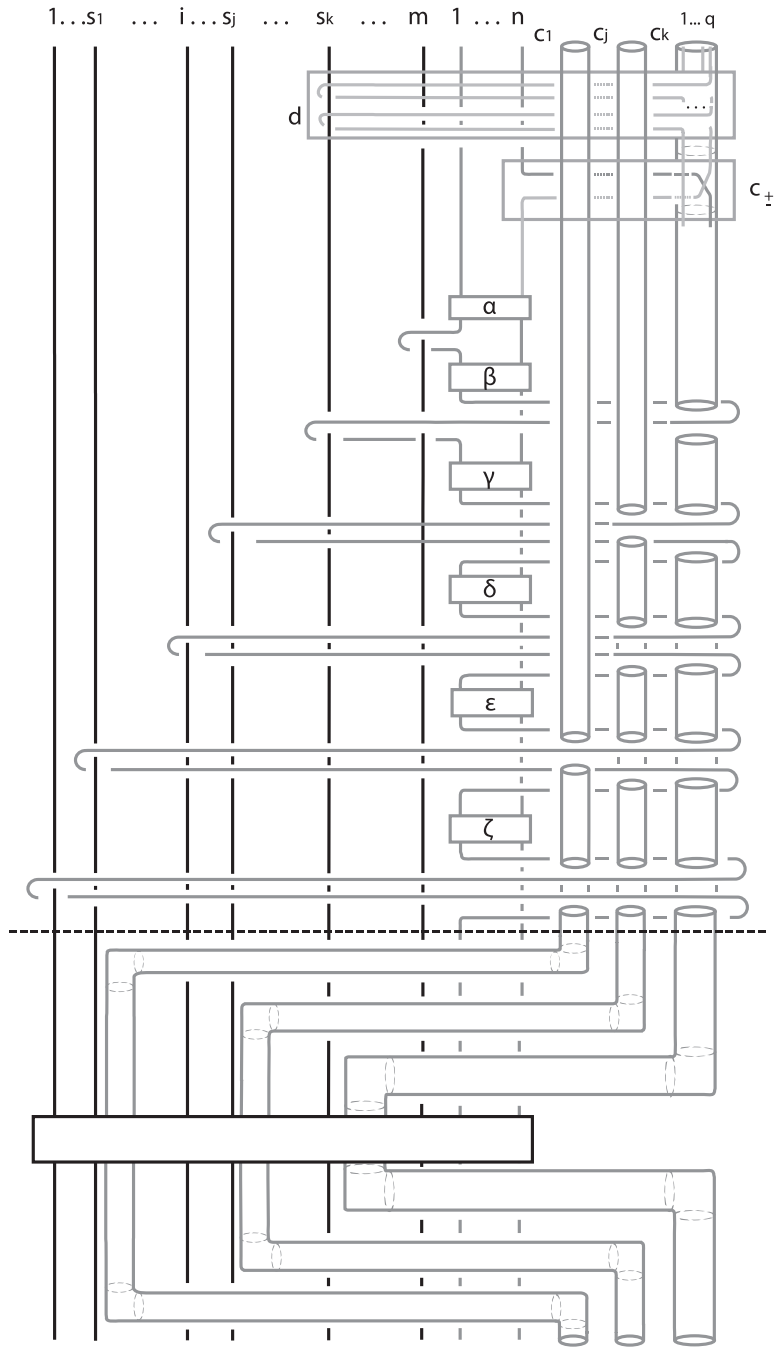}
\end{center}
\caption{ Algebraization of the loop generators and preparation for combing. }
\label{seifcomb5}
\end{figure}


\begin{figure}
\begin{center}
\includegraphics[width=3.3in]{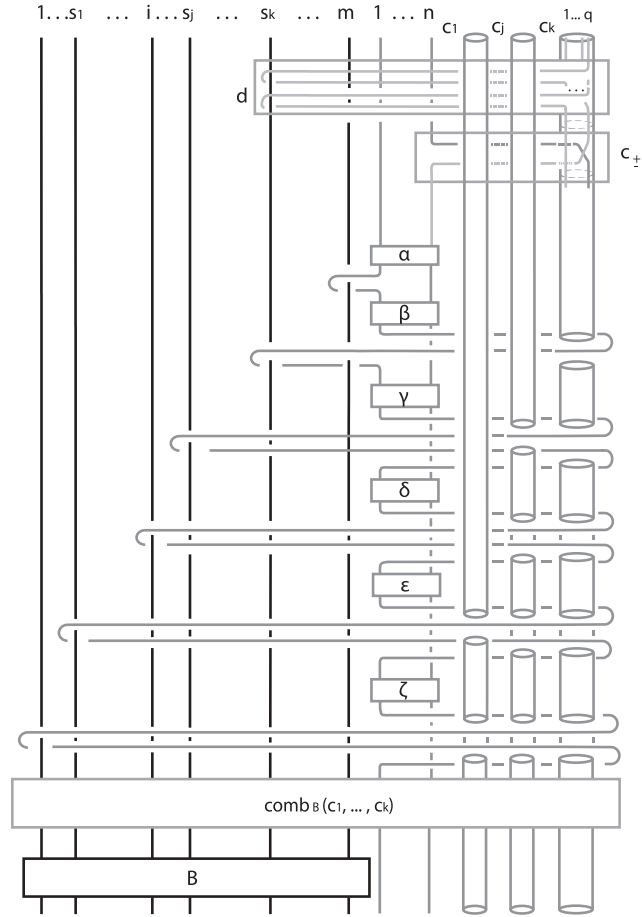}
\end{center}
\caption{ The combing of the cables through $B$. }
\label{seifcomb6}
\end{figure}


\begin{thebibliography}{99}

\bibitem[HL]{HL}{\sc R.~H\"aring-Oldenburg, S.~Lambropoulou},
Knot theory in handlebodies, {\it J. Knot Theory and Ramifications} {\bf 11}, No. 6,
921-943 (2002).

\bibitem[La1]{La1} {\sc S. Lambropoulou}, $L$-moves and Markov theorems, {\it J. Knot Theory Ramifications} {\bf 16} no. 10 (2007), 1-10.

\bibitem[La2]{La2}{\sc S. Lambropoulou}, Braid structures in handlebodies, knot complements and
3-manifolds,  {\it Proceedings of Knots in Hellas '98}, World Scientific Press,
Series of Knots and Everything {\bf 24}, 274-289 (2000).

\bibitem[La3]{La3}{\sc S. Lambropoulou}, Knot theory related to generalized and cyclotomic
Hecke algebras of type {\it B}, {\it J. Knot Theory and its Ramifications} {\bf 8},
              No. 5,  621-658 (1999).

\bibitem[LM]{LM} {\sc L. Moser}, Elementary surgery along a torus knot, {\it Pacific J. Math.} {\bf 38}, 737-745 (1971).

\bibitem[LR1]{LR1}
{\sc S. Lambropoulou, C.P. Rourke} (2006), Markov's theorem in $3$-manifolds, \emph{Topology and its Applications} {\bf 78},
95-122 (1997).

\bibitem[LR2]{LR2} {\sc S. Lambropoulou, C. P. Rourke}, Algebraic Markov equivalence for links in $3$-manifolds,
              {\em Compositio Math.} {\bf 142} (2006) 1039-1062.

\bibitem[Ro]{Ro} {\sc D. Rolfsen}, Knots and Links, \textit{Publish or Perish, Inc., Berkeley, CA} (1976).

\bibitem[Sa]{Sa} {\sc N. Saveliev}, Lectures on the Topology of $3$-Manifolds. An Introduction to the Casson Invariant, {\it De Gruyter Textbook} (1999).

\bibitem[Sk]{Sk} {\sc R. Skora},  Closed braids in 3-manifolds,
             {\it Math. Zeitschrift} {\bf 211} (1992) 173-187.

\bibitem[Su]{Su} {\sc P.A. Sundheim},  Reidemeister's theorem for 3-manifolds,
              {\it Math. Proc. Camb. Phil. Soc.} {\bf 110} (1991)  281-292.

\end{thebibliography}
\end{document}